\documentclass[12pt]{article}
\usepackage[active]{srcltx}
\usepackage[curve]{xypic}
\usepackage{graphicx}

\usepackage[active]{srcltx}
\usepackage{epsfig,amsmath}
\usepackage[english]{babel}
\usepackage{amsmath,amsfonts,amssymb,amscd,color,epsfig,amsthm}

\newtheorem{newthm}{Theorem}
\newtheorem{theorem}{Theorem}[section]
\newtheorem{lemma}[theorem]{Lemma}
\newtheorem{proposition}[theorem]{Proposition}
\newtheorem{corollary}[theorem]{Corollary}

\newtheorem{definition}[theorem]{Definition}

\theoremstyle{remark}
\newtheorem{example}{\bf Example}[section]

\newtheorem{rema}[example]{\bf Remark}

\theoremstyle{plain}

\newtheorem{coro}[theorem]{\bf Corollary}

\numberwithin{equation}{section}

\def\QQQ{{\cal Q}}

\def\ep{\varepsilon}

\def\smm{\smallsetminus}

\def\R{\mbox{$\mathbb R$}}

\def\C{\mbox{$\mathbb C$}}
\def\T{\mbox{$\mathbb T$}}

\def\D{\mathbb D}
\def\Z{\mbox{$\mathbb Z$}}
\def\Q{\QQQ}
\def\N{\mbox{$\mathbb N$}}

\def\lv{ \left(\begin{matrix} }
 \def\rv{\end{matrix}\right)}

\def\P{{\mathbb P}^1}

\def\eqdef{:=}

\def\cal{\mathcal}

\def \eps{\varepsilon}

\def\dz{{\ \rm{d}z}}
\def\dw{{\dw}}

\newcommand{\mylabel}[1]{\label{#1}}

\newcommand{\REFEQN}[1] { \begin{equation}\mylabel{#1} }
\newcommand{\ENDEQN}{\end{equation}}
\newcommand{\REFTHM}[1] { \begin{theorem}\mylabel{#1} }
\newcommand{\ENDTHM}{\end{theorem}}
\newcommand{\REFNTH}[1] { \begin{newthm}\mylabel{#1} }
\newcommand{\ENDNTH}{\end{newthm}}
\newcommand{\REFPROP}[1]{\begin{proposition}\mylabel{#1} }
\newcommand{\ENDPROP}{\end{proposition} }
\newcommand{\REFLEM}[1]{\begin{lemma}\mylabel{#1} }
\newcommand{\ENDLEM}{\end{lemma} }
\newcommand{\REFCOR}[1]{\begin{corollary}\mylabel{#1} }
\newcommand{\ENDCOR}{\end{corollary} }
\newcommand{\Ref}[1]{ (\ref{#1}) }

\def\smm{ {\smallsetminus }}

\def\mystrut{{\rule[-2ex]{0ex}{4.5ex}{}}}

\def\T{{\mathbb T}}

\begin{document}

\title{The dynatomic periodic curves for polynomial $\textbf{z}\mapsto \textbf{z}^\textbf{d}+\textbf{c}$ are smooth and irreducible}

\author{
Yan Gao
\thanks{
Academy of mathematics and systems science, Chinese academy of sciences. (\tt email: gyan@mail.ustc.edu.cn)}
\and Yafei Ou
\thanks{
Universit\'e d'Angers. ( \tt email:  yafei.ou@univ-angers.fr) }
}

\maketitle

\begin{abstract}

 We prove here the smoothness and the irreducibility of the
 periodic dynatomic curves $ (c,z)\in \C^2$ such that $z$ is $n$-periodic for $z^d+c$, where $d\geq2$.

 We use the method provided by Xavier Buff and Tan Lei in \cite{BT} where they prove the conclusion for $d=2$.
 The proof for smoothness is based on elementary calculations on the pushforwards of specific quadratic differentials,
 following Thurston and Epstein, while the proof for irreducibility is a simplified version of Lau-Schleicher's proof
 by using elementary arithmetic properties of kneading sequence instead of internal addresses.

\end{abstract}

\section{Introduction}

For $c\in \C$, set $f_c(z)=z^d+c$, where $d\geq2$. For $n\geq 1$,
define
 $$X_{n}:=\bigl\{(c,z)\in \C^{2}\mid f^{n}_{c}(z)=z,\ (f^{n}_{c})'(z)\ne 1 \text{ and for all } 0<m< n,\ \ f^{m}_c(z)\ne z\bigr\}.$$

The objective of this note is to give an elementary proof of the following results:

\REFTHM{smooth} For every $n\geq 1$, the closure of $X_n$  in $\C^2$ is smooth.\ENDTHM

\REFTHM{irr} For every $n\geq 1$ the closure of $X_n$  in $\C^2$ is irreducible.\ENDTHM

The first example is, as $d=2$
$$X_1=\bigl\{(c,z)\in \C^2\mid z^2+c=z\bigr\}\smm \bigl\{(\dfrac14,\dfrac12)\bigr\}=\{(c,z)\in \C^2\mid c = z-z^2\}\smm \bigl\{(\dfrac14,\dfrac12)\bigr\}$$
and $$\overline{X_1}=\bigl\{(c,z)\in \C^2\mid c=z-z^2\bigr\}\\ .$$

In the case $d=2$, Theorem \ref{smooth} was proved by Douady-Hubbard and Buff-Tan in different methods;
Theorem \ref{irr} was proved by Bousch, Morton, Lau-Schleicher and Buff-Tan with different approaches.

Our approach here to the two Theorems is a generalisation to that used by Xavier Buff and Tan Lei in \cite{BT}, where they prove
the conclusion for $d=2$. To prove Theorem \ref{smooth}, we use elementary calculations on quadratic differentials and
Thurston's contraction principle.  To prove Theorem \ref{irr}, we use a dynamical method by a purely arithmetic argument on kneading sequences(Lemma \ref{key} below).

Section 2 proves the smoothness and Section 3 proves the irreducibility. The two sections can be read
independently.

Acknowlegement. We thank Tan Lei for helpful discussions, and Yang Fei for providing some good pictures

\section{Smoothness of the periodic curves}

For $n\ge 1$, and  $(c,z)\in \C^2$, we say that $z$ is periodic of
{\em period} $n$ for $f_c: z\mapsto z^d+c$, where $d\geq2$, if
 $f_c^{\circ n}(z)=z$ and
for all $0<m<n$,  $f^m_c(z)\ne z$. In this case the {\em multiplier} of $z$ for $f_c$ is defined to be
$[f^n_{c}]'(z)$.

We define
 $$X_{n}:=\bigl\{(c,z)\in \C^{2}\mid  z\text{ is of period $n$ for $f_c$ and of multiplier distinct from } 1\bigr\}.$$
The objective of here  is to give an elementary proof of the following result:

\REFTHM{Smooth} For every $n\geq 1$, the closure $\overline{X_n}$ of $X_n$  in $\C^2$ is smooth.
More precisely, the boundary $\partial {X_n}$ is the finite set of $(c,z)\in \C^2$ such that $z$ is of period $m\le n$ dividing $n$ for $f_c$
 whose multiplier is  of the form $e^{2\pi i u/v}$ with $u,v\ge 1$ co-prime and $v=n/m$. In a neighborhood of a point $(c_0,z_0)\in \overline{X_n}$,
 the
 set $\overline{X_n}$ is locally the graph of a holomorphic map $\left\{\begin{array}{ll} \mystrut c\mapsto z(c) \text{ with $z(c_0)=z_0$} & \text{if } (c_0,z_0)\in X_n\\
 z\mapsto  c(z) \text{ with $c(z_0)=c_0$ and $c'(z_0)=0$} & \text{if } (c_0,z_0)\in \partial X_n\end{array}\right. $.
 \ENDTHM

The idea is to prove that some partial derivative of some defining function of $\overline{X_n}$ is non vanishing.
Following A. Epstein, we will express this derivative as the coefficient of a quadratic differential of the form
$(f_c)_*\Q-\Q$. Thurston's contraction principle gives $(f_c)_*\Q- \Q\ne 0$, therefore the non-nullness of our partial derivative.

\subsection{Quadratic differentials and contraction principle}

A meromorphic quadratic differential (or in short, a quadratic differential)  $ \Q $ on $\C$  takes the form $ \Q = q \dz ^ 2 $ with $ q $ a meromorphic function on $ \C $.

We use $ \Q (\C) $ to denote the set of
meromorphic quadratic differentials on $ \C $ whose poles (if any) are all simple. If $\Q\in \Q(\C)$ and $ U $ is a bounded open subset of $ \C $, the norm
\[\| \Q \| _U \eqdef \iint_U| q| \]
is well defined and finite.

For example $$\|\dfrac{\dz^2}{z}\|_{\{|z|<R\}}=\int_{0}^{2\pi}\!\!\!\!\int_{0}^R \dfrac1r r\,drd\theta =2\pi R\ .$$

For $ f: \C \to \C $  a non-constant polynomial and $\Q=q\dz^2$ a meromorphic quadratic differential on $\C$,  the pushforward $ f_* \Q$ is defined by the quadratic differential
\[ f_* \Q: = Tq \dz ^ 2\quad\text{with}\quad  Tq (z) \eqdef \sum_{f (w) = z} \frac{q (w)}{f '(w) ^ 2}. \]
If $ Q \in \Q (\C) $, then $ f_* \Q \in \Q (\C) $ also.

The following lemma is a weak version of Thurston's contraction principle.

\begin{lemma} [contraction principle] \label{lemme_thurstonpoly} For a non-constant polynomial $f$ and
a round disk $V$ of radius large enough so that
$ U \eqdef f ^{-1} (V) $ is relatively compact in $ V $, we have
\[\| f_* \Q \| _V \leq \| \Q \| _U <\| \Q \| _V,\quad \forall\, \Q\in \Q(\C). \]
\end{lemma}

\proof The strict inequality on the right is a consequence of the fact that $ U $ is
relatively compact in $ V $. The  inequality on the left comes from
\begin{align*}
\|f_* \Q\|_V & =\iint_{z\in V} \left|\sum_{f(w)=z} \frac{q(w)}{f'(w)^2}\right|\ |{\rm d}z|^2 \\
& \leq  \iint_{z\in V} \sum_{f(w)=z}\left| \frac{q(w)}{f'(w)^2}\right| \ |{\rm d}z|^2\\
 & =  \iint_{w\in U} \bigl| q(w)\bigr| \ |{\rm d}w|^2 = \|\Q\|_U.
 \end{align*}
\qed

\begin{coro} \label{qneqq coro_f}
If $ f: \C \to \C $ is a polynomial and if $ \Q \in \Q (\C) $, then
$ f_* \Q \neq \Q $.
\end{coro}

\begin{rema}  Thurston's contraction principal  says that if
$ \Q $ is a meromorphic quadratic differential on $ \P $ and $ f: \P
\to \P $ is a rational function, if one requires $f_*\Q=\Q$ with
$\Q\not=0$, then  $ f $ is necessarily a Latt\`es example.
\end{rema}

The formulas below appeared in \cite{LG} chapter $2$, we write them together as a lemma.
\begin{lemma} [Levin] \label{lemme_fc}
For $f=f_c$, we have
\begin{equation}\label{pushforward} \left\{\begin{array}{ll} \mystrut f_*\left(\dfrac{ \dz^2}{z}\right)=0\vspace{0.1cm}\\
f_*\left(\dfrac{\dz^2}{z-a}\right)=\dfrac{1}{f'(a)}
\left(\dfrac{\dz^2}{z-f(a)}-\dfrac{\dz^2}{z-c}\right) & \text{if } a\ne 0\vspace{0.1cm}\\
f_*\left(\dfrac{\dz^2}{(z-a)^2}\right) =
\dfrac{\dz^2}{(z-f(a))^2}-\dfrac{d-1}{af'(a)}\left(\dfrac{\dz^2}{z-f(a)}-\dfrac{\dz^2}{z-c}\right)
& \text{if } a\ne 0. \end{array}\right.\end{equation}
\end{lemma}

\subsection{Proof of Theorem \ref{Smooth}}

\begin{lemma} [compare with \cite{Mil}] \label{lemme_dzndc}
Given $ z \in \C $, for $ n \geq 0 $ and $d\geq2$, define $ z_n: c
\mapsto f_c ^{\circ n} (z) $ and $\delta_n=f_c'(z_n)=dz_n^{d-1}$.
Then
\[\frac{\dz_n}{dc} = 1+\delta_{n-1}+\delta_{n-1}\delta_{n-2}+\ldots+
\delta_{n-1}\delta_{n-2}\cdots\delta_1.\]
\end{lemma}

\proof From $z_n=z_{n-1}^d+c $, $d\geq2$, we obtain
\[\frac{\dz_n}{dc} = 1 + \delta_{n-1} \frac{\dz_{n-1}}{dc} \quad \text{with} \quad
\frac{\dz_0}{dc} = 0. \] The result follows by induction. \qed

{\noindent\em Proof of Theorem} \ref{Smooth}.

 Let $ P_n (c, z) \eqdef f_c ^{\circ n} (z)-z $ and consider the algebraic curve
 $$ Y_n :=\{(c,z)\in \C^2 \mid P_n (c, z) = 0 \}. $$
 If $ (c, z) \in Y_n $, the point $ z $ is
periodic for $f_c $ of period $m\le n$. Then $m$ divides $n$\footnote{use the formula $0=f_c^{\circ n}(z)-z=
f_c^{\circ km+\ell}(z)-z=f_c^{\circ \ell}(f_c^{\circ km}(z))-z=f_c^{\circ \ell}(z)-z$
and the minimality of $m$ to conclude that $m$ divides $n$.}.
Therefore $Y_n$ is the set of $(c,z)$ such that $z$ is periodic for $f_c$ of period $m\le n$ and $m$ dividing $n$.

As $Y_n$ is a closed subset of $\C^2$, we have $\overline{X_n}\subset Y_n$.

We decompose $Y_n$ into
\begin{eqnarray*}
 Y_n &=& X_n \\ && \sqcup\ \{(c,z)\mid \text{ $z$ is of period $n$ for $f_c$ with multiplier $1$}\} \\
&& \sqcup\  \{(c,z)\mid \text{ $z$ is of period $m$ for $f_c$ with $m<n$ and $m$ dividing $n$}\}
\end{eqnarray*}

We will examine case by case points in $Y_n$, determine  points   in $\overline{X_n}$ and
establish the smoothness of $\overline{X_n}$ at each of these points.

{\noindent\bf Case 1}.
Consider a point $ (c_0, z_0) \in X_n \subset Y_n $.

If $ (c, z) \in Y_n $ is close to
$ (c_0, z_0) \in X_n $, the points of the orbit of $ z $ are close to
points of the orbit of $ z_0 $ and there are therefore at least $ n $ distinct points
in the orbit of $z$. It follows that the  period of $z$ is equal to $ n $. This shows that in a
neighborhood of $ (c_0, z_0) $, the curves $ X_n $ and $ Y_n $ coincide. It
suffices to show that $ Y_n $ is smooth in a neighborhood of
$ (c_0, z_0) $.
As $ [f_{c_0} ^{\circ n}] '(z_0) \neq 1 $, we have
\[\frac{\partial  P_n}{\partial  z}(c_0,z_0)\neq 0.\]
The implicit function theorem implies that $ Y_n $, therefore $X_n$, is smooth in a
neighborhood of $ (c_0, z_0) $.

{\noindent\bf Case 2}.
Now consider a point $(c_0,z_0)\in Y_n$ such that $z_0$ is of period equal to $n$ for $f_{c_0}$
with multiplier $1$.

Fix any $\ell\ge n$ that is a multiple of $n$. And consider $P_\ell$ and $Y_\ell$. We know that
\REFEQN{ell} (c_0,z_0)\in Y_\ell\quad\text{and}\quad[f_{c_0}^{\ell}]'(z_0)=1\ .\ENDEQN

{\noindent\bf Claim}. For any triple $(c_0,z_0,\ell)$ satisfying \Ref{ell}, we have $\dfrac{\partial P_\ell}{\partial c}(c_0,z_0)\neq 0$.

{\noindent\rm Proof}.
For $k\ge 0$, define inductively $z_{k+1}=f_{c_0}(z_{k})$ and define $\delta_k\eqdef f'_{c_0}(z_k)$. We have, by Lemma \ref{lemme_dzndc}
\[\frac{\partial  P_\ell}{\partial  c} (c_0,z_0)= \frac{d}{dc}( f_c^{\circ \ell}(z_0)-z_0)\Bigr|_{c_0}=1+\delta_{\ell-1}+\delta_{\ell-1}\delta_{\ell-2}+\ldots+
\delta_{\ell-1}\delta_{\ell-2}\cdots\delta_1.\]
Now consider the
quadratic differential $ \Q \in \Q (\C) $ defined by
\[\Q(z)\eqdef \sum_{k=0}^{\ell-1}\frac{\rho_k}{z-z_k}\dz^2,\quad\text{with
}\rho_k=\delta_{\ell-1}\delta_{\ell-2}\cdots\delta_k.\]  Applying Lemma \ref{lemme_fc}, and writing $ f $ for
$ f_{c_0} $, we obtain
\[f_*\Q(z)
=\sum_{k=0}^{\ell-1}\frac{\rho_k}{\delta_k}\left(\frac{\dz^2}{z-z_{k+1}}
-\frac{\dz^2}{z-c_0}\right) = \Q(z) -\frac{\partial  P_\ell}{\partial  c}
(c_0,z_0)\cdot\frac{\dz^2}{z-c_0}.\] By Corollary
\ref{qneqq coro_f}, we can not have $ f_* \Q = \Q $. It follows that \[\frac{\partial  P_\ell}{\partial  c } (c_0, z_0) \neq 0. \]
This  ends the proof of the claim.

Now let $\ell=n$, by implicit function theorem,there exists unique locally holomorphic function $c(z)$ with
$f^n_{c(z)}(z)=z, c(z_0)=z_0$ and $c^\prime(z_0)=0 (\text{for}\  \frac{\partial  P_\ell}{\partial  z } (c_0, z_0)= 0)$. Then there is neighborhood $U$ of $(c_0,z_0)$ in $\C^2$ such that
\[Y_n\cap U=\{ (c(z),z) | |z-z_0|<\varepsilon \}.\]
As $z_0$ is a $n$ periodic point of $f_{c_0}$ and the map $z\mapsto [f_{c(z)}^{\circ n}]'(z)$ is holomorphic and {\bf can not
be}\footnote{One can prove that $D:=\{(c,z)\mid f_c^{\circ n}(z)=z,
[f_{c}^{n}]'(z)=1\}$ is finite as follows: Denote by $X(c)$,
resp. $Y(z)$ the resultant of the two polynomials $f_{c}^{\circ
n}(z)-z$ and $[f_{c}^{n}]'(z)-1$ considered as polynomials of
$z$, resp. of $c$. Then $X(c)$ is a polynomial of $c$, resp. $Y(z)$
is a polynomial of $z$. The projection of $D$ to each coordinate
equals the zeros of $X$, resp. of $Y$. As no point of the form
$(0,z)$, $(c,0)$ is in $D$, we have $X(0)\ne 0\ne Y(0)$ so $X$, $Y$
each has finite many roots. As $D\subset (X^{-1}(0)\times \C)\cap
(\C\times Y^{-1}(0))$ we know that $D$ is finite.} constantly $1$, we can choose $\varepsilon$ small enough such that
$z$ is $n$ periodic point of $f_{c(z)}$ with multiplier $\neq 1$ for $|z-z_0|<\varepsilon$. Then
\[U\cap Y_n \setminus \{ (c_0,z_0) \}\subset U\cap X_n\subset U\cap Y_n.\]
It follows $(c_0,z_0)\in \partial X_n$ and $U\cap Y_n$ is a neighborhood of $(c_0,z_0)$ on $\overline{X_n}$.
Then $\overline{X_n}$ is smooth at $(c_0,z_0)$ and parametered locally by $z$.

{\noindent\bf Case 3}. Finally consider $(c_0,z_0)\in Y_n$ so that $z_0$ is of period $m< n$ for $f_{c_0}$
with $m$ dividing $n$.

Note that $Y_m\subset Y_n$.

If  $[f_{c_0}^{n}]'(z_0)\ne 1$ then $[f_{c_0}^{m}]'(z_0)\ne 1$. By the existence and the unicity of
 the implicit function
theorem the local solutions of $f_c^n(z)-z=0$ and $f_c^m(z)=z$ coincide,
that is, $Y_m$ and $Y_n$ coincide locally. So at point $(c_0,z_0)$, $Y_n$ is locally the graph of a  holomorphic function $z(c)$ with $z(c_0)=z_0$ and $z(c)$ is $m$ periodic point of
$f_c$. It follows that $(c_0, z_0)\notin \overline{X_n}$.

If $[f_{c_0}^{n}]'(z_0)=1$ and $[f_{c_0}^{m}]'(z_0)=1$,
 then both triples $(c_0,z_0,m)$ and $(c_0, z_0,n)$ satisfy \Ref{ell}. The claim in Case 2 and implicit
 function theorem imply that $Y_m$ and $Y_n$ again coincide in a neighborhood of $(c_0,z_0)$.
 For the same reason as above, $(c_0,z_0)\notin \overline{X_n}$.

Set $\rho:= [f_{c_0}^{m}]'(z_0)$. We consider now the only remaining case $\rho\ne 1$ and $\rho^{n/m}=[f_{c_0}^{n}]'(z_0) =1$.

 Fix any integer $s\ge 2 $  such that $\rho^s=1$. Let $c_\ast$ be any point outside Mandelbrot set, then each zero point of $f_{c_\ast}^m(z)-z$ is simple. It follows that 
 $f_{c_\ast}^m(z)-z$ divides $f_{c_\ast}^{\circ ms}(z)-z$. Since $c_\ast$ is any point outside Mandelbrot set, the polynomial $f_c^m(z)-z$ must divides $f_{c}^{ms}(z)-z$.
 Let $ P (c, z) $ be the polynomial defined by the equation:
\begin{equation} \label{eq: CRD}
f_c^{\circ ms}(z)-z = \bigl(f_c^{\circ m}(z)-z\bigr)\cdot P(c,z).
\end{equation}
  {\noindent\bf Claim.} Let $Z_s:=\{(c,z)\mid P(c,z)=0\}$. Then $(c_0,z_0)\in Z_s$ and there is a neighborhood $V$ of $(c_0,z_0)$ in $\C^2$ such that
\[Z_s\cap V=\{ (c(z),z) | |z-z_0|<\varepsilon_0, c(z) \text{\ is holomorphic with}\ c(z_0)=c_0 \text{\ and} \ c^\prime(z_0)=0\}.\]

Proof. We will prove at first that  the map $ z \mapsto f_{c_0} ^{ ms} (z)-z $ has a zero of order at least $3$
at $z_0$. Define $F(z)=f^m(z+z_0)-z_0$, then it is equivalent to show the function $F^{s}(z):=f_{c_0}^{ms}(z+z_0)-z_0$ has a local expansion $z+O(z^3)$ at $0$.
We have $F(z)=\rho z + az^2+O(z^3)$
in a neighborhood of $0$. One checks by induction
$$\forall\,k\ge 1,\quad F^{\circ k}(z)=\rho^k z+a\rho^{k-1}(1+\rho+\rho^2+\cdots +\rho^{k-1})z^2+O(z^3)\ .$$
But $\rho\ne 1$ and $\rho^s=1$, it follows that $1+\rho+\rho^2+\cdots +\rho^{s-1}=0$ and $F^{\circ s}(z)=z+O(z^3)$.

Since
  $ z \mapsto f_{c_0} ^{\circ m} (z)-z $ has a simple zero, we see from \Ref{eq: CRD} that
$ z \mapsto P (c_0, z) $ has a zero of order at least $2$ at $z_0$. Therefore $(c_0,z_0)\in Z_s$ and
\REFEQN{graph} \frac{\partial P}{\partial z} (c_0, z_0) = 0. \ENDEQN
We proceed now to prove
\begin{equation} \label{partial} \frac{\partial P}{\partial c} (c_0, z_0) \neq 0. \end{equation}

This will be down in two steps:

{\noindent\bf Step 1}. Let $ Q (c, z) \eqdef f_c ^{\circ m} (z)-z  $. We have
\[Q (c_0, z_0) = 0 \quad \text{and} \quad \frac{\partial Q}{\partial z} (c_0, z_0) = \rho-1 \neq 0. \] According to the implicit function theorem, there is a germ of a holomorphic function
$ \zeta : (\C, c_0) \to (\C, z_0) $ with $ Q \bigl (c, \zeta (c) \bigr) = 0 $. In other words, $ \zeta (c) $ is a periodic point of period $ m $ for
$  f_c$. Let $ \rho_c$ denote the multiplier of  $ \zeta (c) $ for $f_c$ and  set
\[\dot \rho \eqdef \frac{d \rho_c}{dc} \big| _{c_0}. \]

\begin{lemma}
We have
\[\frac{\partial P}{\partial c} (c_0, z_0) = \frac{s \cdot \dot \rho}{\rho (\rho-1)}. \]
\end{lemma}

\begin{proof}
Differentiating the equation (\ref{eq: CRD}) with respect to $ z $, and then evaluating at $ \bigl (c, \zeta (c) \bigr) $, we get:
\[\rho_c^s-1 = (\rho_c-1)\cdot P\bigl(c,\zeta(c)\bigr) + \underset{=0}{\underbrace{\bigl(f_{c}^{m}\bigl(\zeta(c)\bigr)-\zeta(c)\bigr)}}\cdot \frac{\partial P}{\partial z}\bigl(c,\zeta(c)\bigr) =  (\rho_c-1)\cdot P\bigl(c,\zeta(c)\bigr).\]
Setting
\[R(c)\eqdef P\bigl(c,\zeta(c)\bigr) = \frac{\rho_c^s-1}{\rho_c-1}, \]
we have
\[R'(c_0) =\frac{\partial P}{\partial c}(c_0,z_0) +\underset{=0}{\underbrace{ \frac{\partial P}{\partial z}(c_0,z_0) }}\cdot \zeta'(c_0) =
 \frac{\partial P}{\partial c}(c_0,z_0) .\]
Using $ \rho ^ s = 1 $ and $ \rho ^{s-1} = 1 / \rho $, we deduce that
\[\frac{\partial P}{\partial c}(c_0,z_0) = \frac{d}{dc}\left(\frac{\rho_c^s-1}{\rho_c-1}\right)\Big|_{c_0} =\left( \frac{s\rho^{s-1}}{\rho-1}-
\frac{\rho^s-1}{(\rho-1)^2}\right)\frac{d\rho_c}{dc}\Big|_{c_0} = \frac{s\cdot \dot \rho}{\rho(\rho-1)}.\qedhere\]
\end{proof}

{\noindent\bf Step 2.}   $\dot \rho \neq 0$. The proof of this fact will be postponed   to the following section $2.3$ using quadratic differential with
double poles
(see also \cite{DH} for a parabolic implosion approach).

This ends the proof of \Ref{partial}, as well as the proof of the claim by combining
  \Ref{partial} and  the implicit function theorem plus the observation that $(c_0,z_0)\in Z_s$.

Write now $\rho=e^{2\pi i u/v}$ with $u,v$ co-prime and $v>0$. Then any $s$ satisfying $\rho^s=1$
takes the form $s=kv$ for some integer $k\ge 1$.
With the same reason as that of existence of polynomial $P(c,z)$ in (\ref{eq: CRD}), there are polynomials $g, h$ such that
$$f_c^{\circ ms}(z)-z= f_c^{\circ mkv}(z)-z=( f_c^{\circ mv}(z)-z) g(c,z)= ( f_c^{\circ m}(z)-z) h(c,z)g(c,z)\ .$$
By definition we have $Z_s=\{(c,z)\mid g(c,z)h(c,z)=0\}\supset \{(c,z)\mid h(c,z)=0\}=Z_v$.
By the claim in Case 3, we conclude that $Z_s$ and $Z_v$ coincide in a neighborhood of $(c_0, z_0)$ as the graph
of a single holomorphic function $c(z)$ with vanishing derivative at $z_0$.

{\noindent\bf Remark}:(1)
If necessary, we can decrease $\varepsilon_0$ in claim of case 3 such that
  $f^m_{c(z)}(z)-z\neq 0$ for $0<|z-z_0|<\varepsilon_0$. Otherwise, there exist a sequence $\{z_k\}$ with
  $z_k\rightarrow z_0$ ( correspondingly, $c_k:=c(z_k)\rightarrow c_0$ ) such that
  $f^m_{c_k}(z_k)-z_k= 0$ and $ h(c_k,z_k)g(c_k,z_k)=0$. It follows that $ [f_{c_k}^{\circ ms}]^\prime(z_k)-1=0$, that is, $\{c_k\}$ is a sequence of parabolic parameter with period of parabolic orbit less than $m$ converging to $c_0$. It is impossible.

  \qquad\ (2) \qquad$Z_s=( Y_n\setminus Y_m )\cup\{(c_0,z_0)\},\ X_n\subset Y_n\setminus Y_m$

\begin{lemma} \label{111}
There exists $0<\varepsilon_1<\varepsilon_0$ such that $z$ is $mv$ periodic point of $f_{c(z)}$ with multiplier$\neq 1$ for $0<|z-z_0|<\varepsilon_1$. $c(z)$ is defined in the claim of case 3.
\end{lemma}

\begin{proof}
Note that $P(c(z),z)=0$ implies $z$ is periodic point of $f_{c(z)}$ with period less than $ms$.
As $(f^m_{c_0})^\prime(z_0)=\rho=e^{2\pi iu/v}$, by lemma \ref{near parabolic} below, when $c$ is close enough to $c_0$, the orbit of $f_{c_0}$ containing $z_0$ splits into two periodic orbit
of $f_c$ with period $m$ and $mv$. Then we can choose $\varepsilon_1<\varepsilon_0$ such that $z$ belongs to one of the
two splitted orbits of $f_{c(z)}$ for $0<|z-z_0|<\varepsilon_1$. By remark (1), the period of $z$ under $f_{c(z)}$ must
be $mv$. The parabolic parameter in $M_d$ with period of parabolic point less than a fixed number are finite, so we can decrease $\varepsilon_1$ if necessary, such that $c(z)$ is not parabolic parameter for $0<|z-z_0|<\varepsilon_1$.
\end{proof}

Now let $V_1$ be a neighborhood of $(c_0,z_0)$ in $\C^2$ with property that \[V_1\cap Z_s=V_1\cap Z_v=\{(c(z),z) | |z-z_0|<\varepsilon_1\}.\]

If $n=mv$, by lemma \ref{111} and remark (2), we have\[(V_1\cap Z_v)\setminus\{(c_0,z_0)\}\subset V_1\cap X_n\subset V_1\cap(Y_n\setminus Y_m)=(V_1\cap Z_v)\setminus\{(c_0,z_0)\}.\] It follows $(c_0,z_0)\in\partial X_n$ and $\overline{X_n}$ coincides with $Z_v$ at neighborhood of $(c_0,z_0)$. Then $\overline{X_n}$ is smooth at point $(c_0,z_0)$

If $n=mvk$ for some $k>1$
\[(V_1\cap Z_s)\setminus\{(c_0,z_0)\}\subset V_1\cap(Y_n\setminus Y_m)=(V_1\cap Z_s)\setminus\{(c_0,z_0)\}.\] Then $V_1\cap Z_s$ is the neighborhood of $(c_0,z_0)$ in $(Y_n\setminus Y_m)\cup\{(c_0,z_0)\}$. For $X_n\subset Y_n\setminus Y_m$ and $X_n\cap (V_1\cap Z_s)=\emptyset$, we have $(c_0,z_0)\notin \overline{X_n}$
 \qed

\subsection{Quadratic differentials with double poles}

Set $f:=f_{c_0}$,
\[z_k\eqdef f^k(z_0),\quad
\delta_k\eqdef dz_k^{d-1}=f'(z_k),\quad \zeta_k(c)\eqdef f_c^{\circ
k}\bigl(\zeta(c)\bigr)\quad \text{and}\quad \dot\zeta_k\eqdef
\zeta_k'(c_0).\] Then
\[\zeta_{k+1}(c) = f_c\bigl(\zeta_k(c)\bigr)\quad \text{and}\quad \zeta_m=\zeta_0.\]
Since \[\delta_0\delta_1\cdots\delta_{m-1}= \rho\neq 0,\]
there is a unique $ m$-tuple $ (\mu_0, \ldots, \mu_{m-1}) $ such that
\[\mu_{k+1}= \frac{\mu_k}{dz_k^{d-1}} - \frac{d-1}{dz_k^d},\]
where the indices are considered  to be modulo $ m$.

Now consider the quadratic differential $ \Q $ (with double poles) defined by
\[\Q\eqdef  \sum_{k=0}^{m-1} \left(\frac{1}{(z-z_k)^2} + \frac{\mu_k}{z-z_k}\right)\dz^2.\]

\begin{lemma} [Compare with \cite{LG}]
We have
\[f_*\Q = \Q - \frac{\dot\rho}{\rho} \cdot \frac{\dz^2}{z-c_0}.\]
\end{lemma}

\begin{proof}
By construction of $\Q$ and  the calculation of $f_*\Q$ in Lemma \ref{lemme_fc},  the polar parts of
$ \Q $ and $f_*\Q$ along the cycle of $z_0$ are identical.  But
$f_*\Q$ has  an extra simple pole at the critical value $c_0$ with coefficient
\[\sum_{k=0}^{m-1}\left(-\frac{\mu_k}{dz_k^{d-1}} + \frac{d-1}{dz_k^d} \right)=-\sum_{k=0}^{m-1}\mu_{k+1}.\]
We need to show that this coefficient is equal to $-\frac{\dot \rho}{\rho}$.

Using $\zeta_{k+1}(c)=\zeta_k(c)^d+c$, we get
\[\dot \zeta_{k+1} = dz_k^{d-1}\dot \zeta_k+1.\]
It follows that
\[\dot\zeta_{k+1}\mu_{k+1}-\mu_{k+1} = dz_k^{d-1} \dot \zeta_k \mu_{k+1}= \dot \zeta_k \mu_k - \frac{(d-1)\dot \zeta_k}{z_k}.\]
Therefore
\[ \sum_{k=0}^{m-1} \mu_{k+1} = \sum_{k=0}^{m-1}\left(
\dot \zeta_{k+1}\mu_{k+1}-\dot \zeta_k \mu_k+ \frac{(d-1)\dot
\zeta_k}{z_k}\right) = (d-1) \sum_{k=0}^{m-1} \frac{\dot
\zeta_k}{z_k} = \frac{\dot \rho}{\rho},\] where last equality is
obtained by evaluating at $ c_0 $ of the logarithmic derivative of
\[\rho_c\eqdef  \prod_{k=0}^{m-1}d\zeta_k^{d-1}(c).\qedhere\]
\end{proof}

\begin{lemma} [Epstein\cite{ep}]
We have $ f_* \Q \neq \Q $.
\end{lemma}

\begin{proof}
The proof rests again on the  contraction principle, but we can not apply directly Lemma \ref{lemme_thurstonpoly} since $ \Q $ is not integrable near the cycle $\left<z_0,\ldots,z_{m-1}\right>$.
Consider a sufficiently large round disk $ V $ so  that $U\eqdef f^{-1}(V)$ is relatively compact in $ V $. Given $ \eps> 0 $, we set
\[V_\eps\eqdef   \bigcup_{k=1}^{m} f^k\bigl(D(z_0,\eps)\bigr)\quad \text{and}\quad U_\eps\eqdef f^{-1}(V_\eps).\]
When $ \eps $ tends to $ 0$, we have
\[\|f_* \Q\|_{V-V_\eps}\leq \|\Q\|_{U-U_\eps} = \|\Q\|_{V-V_\eps} -\|\Q\|_{V-U} + \|\Q\|_{V_\eps-U_\eps} - \|\Q\|_{U_\eps-V_\eps}.\]
If we had $ f_* \Q = \Q $, we would have
\[0< \|\Q\|_{V-U} \leq  \|\Q\|_{V_\eps-U_\eps}.\]
However, $\|\Q\|_{V_\eps-U_\eps}$ tends to $0$ as $\eps$ tends to $0$, which is a contradiction. Indeed, $\Q=q(z)dz^2$, the meromorphic function $q$ is equivalent to 
$\dfrac{1}{(z-z_0)}$ as $z$ tends to $z_0$. In addition, since the multiplier of $z_0$ has modulus $1$,
\[D(z_0,\eps)\subset U_{\eps}-V_{\eps}\subset D(z_0,\eps^\prime)\quad\text{with}\quad \frac{\eps^\prime}{\eps}\overset{\eps\rightarrow0}{\longrightarrow}1.\]
Therefor,
\[\|\Q\|_{V_\eps-U_\eps}\leq\int_{0}^{2\pi}\int_{\eps}^{\eps^\prime}\frac{1+o(1)}{r^2}rdrd\theta=2\pi(1+o(1))\log\frac{\eps^\prime}{\eps}\overset{\eps\rightarrow0}{\longrightarrow}0\]
\end{proof}

The fact $\dot \rho\ne 0$ follows from the above two lemmas.

\section{The irreducibility of the periodic curves}

Recall that $f_c$ denote the polynomial $z\mapsto z^d+c$, where
$d\geq2$, and we have defined
 $$X_{n}:=\bigl\{(c,z)\in \C^{2}\mid f^n_{c}(z)=z,\ [f^n_{c}]'(z)\ne 1 \text{ and for all } 0<m< n,\ \ f^m_c(z)\ne z\bigr\}.$$

The objective here is to prove:

\REFTHM{main} For every $n\geq 1$, the set $X_n$ is connected.\ENDTHM

It follows immediately that the closure of $X_n$ in $\C^2$ is irreducible.

\subsection{Kneading sequences}

Set $\T=\R/\Z$ and let $\tau:\T\to \T$ be the angle map
\[\tau: \T\ni \theta\mapsto d\theta\in  \T, \, d\geq2.\]
We shall often make the confusion between an angle $\theta\in \T$
and its representative in $[0,1[$. In particular, the angle
$\theta/d\in \T$ is the element of $\tau^{-1}(\theta)$ with
representative in $[0,1/d[$ and the angle $(\theta+(d-1))/d$ is the
element of $\tau^{-1}(\theta)$ with representative in
$[{(d-1)}/d,1[$.

Every angle $\theta\in \T$ has an associated kneading sequence
$\nu(\theta)=\nu_1\nu_2\nu_3\ldots$ defined by
$$\nu_k=\left\{\begin{array}{ll} 1 & \text{if } \tau^{
k-1}(\theta)\in \ \left]\dfrac \theta d,\dfrac{\theta+1}d\right[,\\
2 & \text{if } \tau^{
k-1}(\theta)\in \ \left]\dfrac {\theta+1} d,\dfrac{\theta+2}d\right[,\\
.\\
.\\
.\\
d-1 & \text{if } \tau^{
k-1}(\theta)\in \ \left]\dfrac {\theta+(d-2)} d,\dfrac{\theta+(d-1)}d\right[,\\
0 & \text{if } \tau^{ k-1}(\theta)\in \T\smm\left[\dfrac \theta
d,\dfrac{\theta+(d-1)}d\right],
\\ \star & \text{if } \tau^{ k-1}(\theta)\in \left\{\dfrac \theta d,\dfrac{\theta+1}d,...,\dfrac{\theta+(d-2)}d,\dfrac{\theta+(d-1)}d\right\}.\end{array}\right. $$
For example,

\begin{itemize}
\item as $d=3$, $\nu(\dfrac17)=\overline{12102\star}$ \quad\text{and}\quad $\nu(\dfrac{27}{28})=\overline{22200\star}$;
\end{itemize}

\begin{figure}[htbp]

\centering

\includegraphics[width=11.5cm]{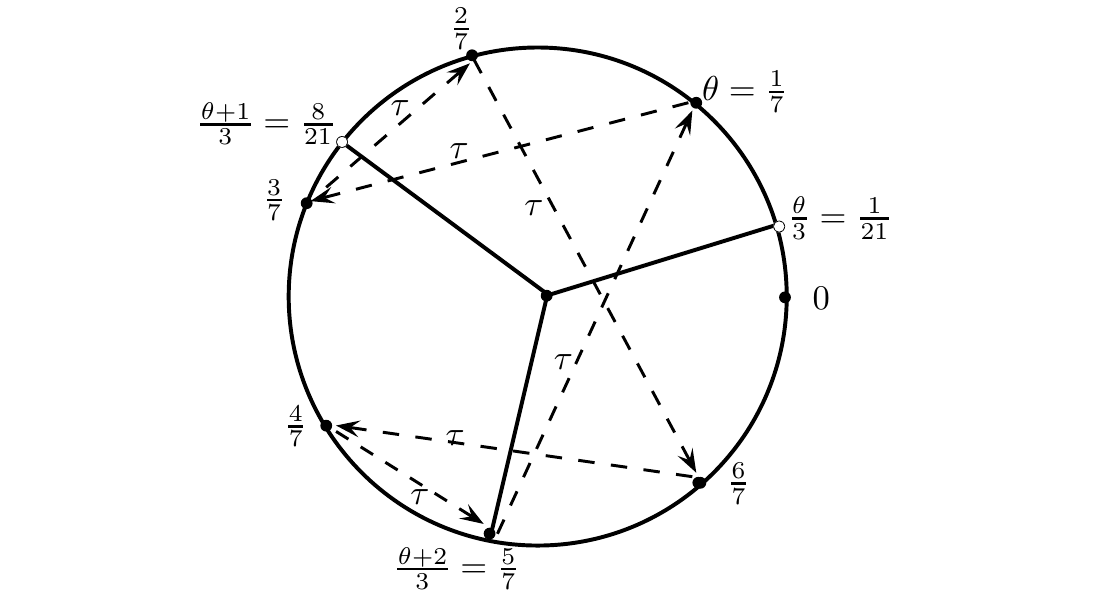}

\caption{As $d=3$, the kneading sequence of $\theta=1/7$ is $\nu(1/7)=\overline{12102\star}$} \label{fig1}

\end{figure}

We shall say that an angle $\theta\in \T$, periodic under $\tau$, is
{\em maximal in its orbit} if its representative in $[0,1)$ is
maximal among the representatives of $\tau^j(\theta)$ in $[0,1)$ for
all $j\geq 1$. If the period is $n$ and the d-expansion ($d\geq2$)
of $\theta$ is $.\overline{\ep_{1}\ldots\ep_{n}}$, then $\theta$ is
maximal in its orbit if and only if the periodic sequence
$\overline{\ep_{1}\ldots\ep_{n}}$ is maximal (in the lexicographic
order) among its shifts. For example, as $d=4$,
$\dfrac{5}{31}=.\overline{02211}$ is not maximal in its orbit but
$\dfrac{20}{31}=.\overline{22110}$ is maximal in the same orbit.

The following lemma indicates cases where the $d-$expansion
$(d\geq2)$ and the kneading sequence coincide.

\begin{lemma}[Realization of kneading sequences]\label{key}
Let $\theta\in \T$ be a periodic angle which is maximal in its orbit
and let $.\overline{\ep_{1}\ldots\ep_{n}}$ be its d-expansion
$(d\geq2)$. Then, $\ep_n\in\{0,1,2,\ldots,d-2\}$ and the kneading
sequence $\nu(\theta)$ is equal to
$\overline{\ep_{1}\ldots\ep_{n-1}\star}$.
\end{lemma}

For example,

\begin{itemize}

\item as $d=3$
\qquad\qquad$\dfrac{13}{14}=.\overline{221001}\quad\text{and}\quad \nu(\theta)=\overline{22100\star}.$
 \item as $d=4$
\qquad\qquad$\dfrac{28}{31}=.\overline{32130}\quad\text{and}\quad \nu(\theta)=\overline{3213\star}.$

\end{itemize}

\begin{figure}[htbp]

\centering

\includegraphics[width=10.0cm]{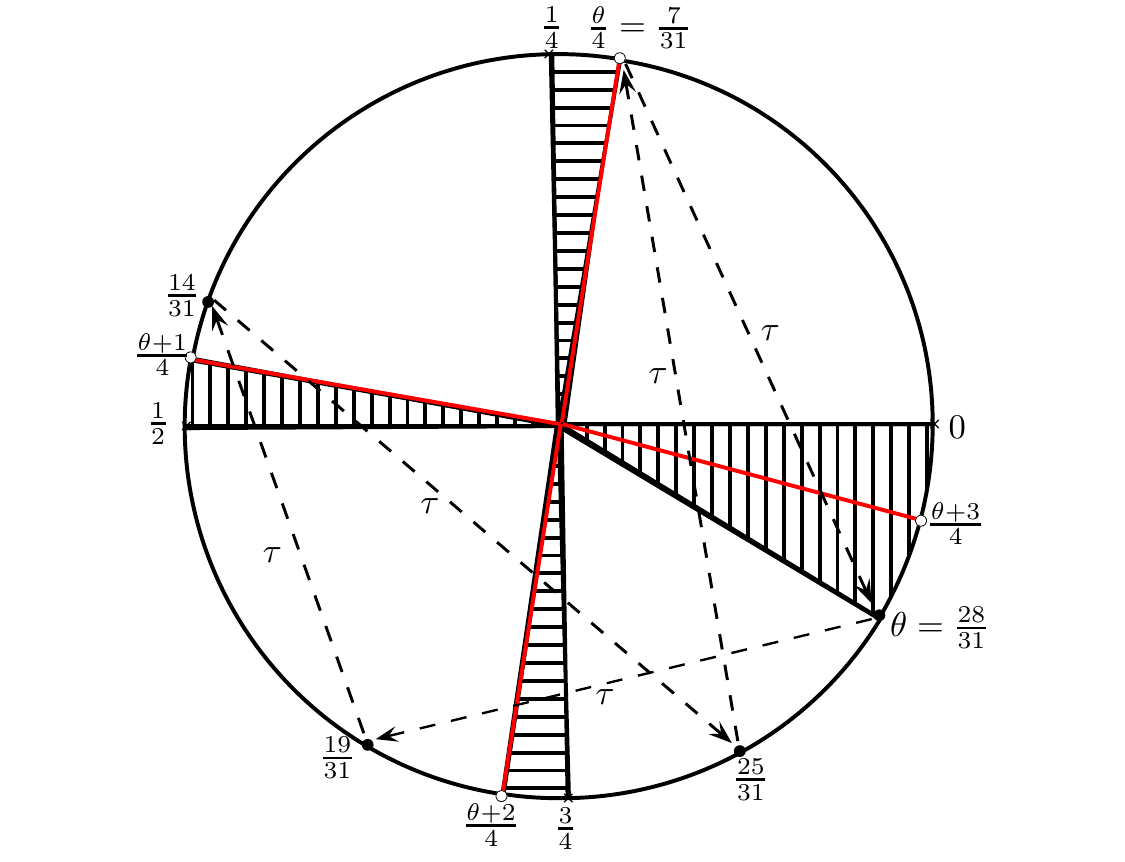}

\caption{As $d=4$, the kneading sequence of $\theta=28/31$ is $\nu(28/31)=\overline{3213\star}$} \label{fig2}

\end{figure}

\begin{proof} Since
$\theta$ is maximal in its orbit under $\tau$, the orbit of $\theta$
is disjoint from $\big]\dfrac{\theta}{d},\dfrac1d\big]\bigcup
 \big]\dfrac{\theta+1}{d},\dfrac2d\big]\\
 \bigcup... \bigcup \big]\dfrac{\theta+(d-2)}{d},\dfrac{d-1}d \big]\bigcup
\big]\theta,1\big]$. It follows that the orbit $\tau^{
j}(\theta)$, $j=0,1,\ldots,n-2$ have the same itinerary relative to
the two partitions
$\T-\big\{0,\dfrac1d,\dfrac2d,\ldots,\dfrac{d-2}d,\dfrac{d-1}d\big\}$
and $\T-\big\{\dfrac{\theta}d,
\dfrac{\theta+1}d,\ldots, \dfrac{\theta+(d-2)}d,\dfrac{\theta+(d-1)}d\big\}$ (see Figure \ref{fig2}).
The first one gives the d-expansion ($d\geq2$) whereas the second
gives the kneading sequence. Therefore, the kneading sequence of
$\theta$ is $\overline{\ep_{1}\ldots\ep_{n-1}\star}$. Since
$\tau^{n-1}(\theta)\in \tau^{-1}(\theta)=\{\dfrac \theta d,
\dfrac{\theta+1}d,\ldots,\dfrac{\theta+(d-1)}d\}$
and since $\dfrac{\theta+(d-1)}{d}\in \ ]\theta,1]$, we must have
$\tau^{n-1}(\theta)=\big\{\dfrac \theta d, \dfrac {\theta+1}
d,\ldots,\dfrac {\theta+(d-2)} d \big\} < \dfrac{d-1}d$. So $\ep_n$,
as the first digit of $\tau^{n-1}(\theta)$, must be in $\{0,1,
2,\ldots,d-2\}$.
\end{proof}

\subsection{Cyclic expression of kneading sequence}

 $X=\{0,1,\ldots, d-1\}(d\geq 2)$ is an alphabet. $X^\star$ is the set of all sequence of symbols from $X$ with finite
 length, that is, \[X^\star=\{\nu_1\ldots\nu_t| \nu_i\in X,t\in\N^\star \}.\] The element of $X^\star$ is called word, its length is denoted by $|\cdot|$. For any $w\in X^\star$, $w$ can be written as $u^n:=\underbrace{u\ldots u}_n$ with
 $u\in X^\star$ and $n\geq 1$.

 For example: \qquad$121212=12^3,\quad 1234=1234$.

 \begin{definition}
 A word is called primitive if it is not the form $u^n$ for any $n>1, u\in X^\star$.
 \end{definition}

The following lemma is a basic result about primitive words due to F.W.Levi. One can refer to \cite{KM} for the proof.

\begin{lemma}[F.W.Levi]\label{primitive root1}
For each $w\in X^\star$, there exists an unique primitive word $a(w)$ such that $w=a(w)^n$ for some $n\geq1$.
\end{lemma}

$a(w)$ is called the primitive root of $w$, this lemma means the primitive root of a word is unique. Let $w$ be a word,
we denote by $L_w$ the set of all words different from $w$ only at the last digit.

\begin{lemma}\label{primitive root2}
If $w$ is a non-primitive word, then any word in $L_w$ is primitive.
\end{lemma}

\begin{proof}
As $w$ is not primitive, then $w=a^m$ where $a$ is the primitive root of $w$ and $m>1$.  $w^\prime$ is any element of $L_w$, then $w^\prime=a^{m-1}a^\prime$ for some $a^\prime\in L_a$.
Now assume $w^\prime$ is not primitive, then $w^\prime=z^n$ where $z$ is the primitive root of $w^\prime$ and $n>1$.
Obviously $|z|\neq|a|$.

If $|z|<|a|$, then $n>m\geq2$ and $a=zb$ for some $b\in X^\star$.
\[a^{m-1}a^\prime=z^n\Longrightarrow za^{m-1}a^\prime=a^{m-1}a^\prime z\Longrightarrow za^{m-1}a^\prime=zba^{m-2}a^\prime z\Longrightarrow\]
\[\exists v\in X^\star, s.t\ a=bv,|v|=|z|\Longrightarrow a^{m-1}bv^\prime=ba^{m-2}a^\prime z (a^\prime=bv^\prime)\Longrightarrow \]

\qquad\quad$v^\prime=z\text{ and}\ a^{m-1}b=ba^{m-2}a^\prime\Longrightarrow a^{m-2}bvb=ba^{m-2}a^\prime\Longrightarrow a^\prime=vb$.

\qquad\quad It is a contradiction to $a=zb$.

If $|z|>|a|$, then there exists $z^\prime\in L_z$ such that $z^{n-1}z^\prime=a^m=w$ with $m>n\geq2$. It reduces to the
case above.

\end{proof}

Now, let $\theta$ be a periodic angle with period $n\geq2$. $\nu(\theta)$ is the kneading sequence of $\theta$.
\begin{definition}
If there is a word $w=\nu_1\ldots\nu_t$ such that $\nu(\theta)=\overline{w^{s-1}w_\star}:=\overline{\underbrace{w\ldots w}_{s-1}w_\star}$, where $w_\star=\nu_1\ldots\nu_{t-1}\star$ and $t$ is a proper factor of $n$ with $ts=n$, then $\nu(\theta)$ is called cyclic,
otherwise $\nu(\theta)$ is called acyclic.
\end{definition}

\begin{definition}\label{expression}
$\nu(\theta)=\overline{w^{s-1}w_\star}$ is cyclic. If $w$ is a primitive word, we call $\overline{w^{s-1}w_\star}$ a cyclic expression of $\nu(\theta)$.
\end{definition}

The following proposition is a corollary of Lemma \ref{primitive root1} and \ref{primitive root2}.

\begin{proposition} \label{cyclic}
If $\nu(\theta)$ is cyclic, then its cyclic expression is unique.
\end{proposition}

\begin{proof}
Assume $\overline{w^{s-1}w_\star}$ and $\overline{u^{l-1}u_\star}$ are two cyclic expression of $\nu(\theta)$ where
$w=\nu_1\ldots\nu_t$ and $u=\epsilon_1\ldots\epsilon_m$. If $\nu_t=\epsilon_m$, then $w^s=u^l$. By Lemma \ref{primitive root1}, we have $w=u$. If $\nu_t\neq\epsilon_m$, then $w^s=u^{l-1}u^\prime$ with some $u^\prime\in L_u$, but this is a contradiction to Lemma \ref{primitive root2}.
\end{proof}

\subsection{Filled-in Julia sets and the Multibrot set}

Let us recall some results about filled-in Julia set and Multibrot set that will be used following. These can be found in \cite{DH}, \cite{Mil}\ and \cite{DE}.

For $c\in \C$, we denote by $K_c$ the filled-in Julia set of $f_c$, that is the set of points $z\in \C$ whose orbit under $f_c$ is bounded. We denote by $M_{d}$ the Multibrot set for $f_{c}(z)=z^{d}+c$, that is the set of parameters $c\in \C$ for which the critical point $0$ belongs to $K_c$.

If $c\in M_{d}$, then $K_c$ is connected. There is a conformal isomorphism $\phi_c:\C\smm \overline K_c\to \C\smm \overline \D$ which satisfies $\phi_c\circ f_c=  \big(\phi_c\big)^{d}$ and $\phi_c^\prime(\infty)=1$. The dynamical ray of angle $\theta\in \T$ is
\[R_c(\theta):=\bigl\{z\in \C\smm K_c\mid \arg\bigl(\phi_c(z)\bigr)=2\pi\theta\bigr\}.\]
If $\theta$ is rational, then as $r$ tends to $1$ from above, $\phi_c^{-1}(r{\rm e}^{2\pi i\theta})$ converges to a point $\gamma_c(\theta)\in K_c$. We say that $R_c(\theta)$ lands at $\gamma_c(\theta)$. We have
$f_c\circ \gamma_c=  \gamma_c\circ \tau$ on $\mathbb{Q}/\Z$.
In particular, if $\theta$ is periodic under $\tau$, then $\gamma_c(\theta)$ is periodic under $f_c$. In addition, $\gamma_c(\theta)$ is either repelling (its multiplier has modulus $>1$) or parabolic (its multiplier is a root of unity).

If $c\notin M_{d}$, then $K_c$ is a Cantor set. There is a conformal isomorphism $\phi_c:U_c\to V_c$ between neighborhoods of $\infty$ in $\C$, which satisfies  $\phi_c\circ f_c= \big( \phi_c\big)^{d}$ on $U_c$. We may choose $U_c$ so that $U_c$ contains the critical value $c$ and  $V_c$ is the complement of a closed disk.  For  each $\theta\in \T$,  there is an infimum $r_c(\theta)\geq 1$ such that $\phi_c^{-1}$ extends analytically along $R_0(\theta)\cap \bigl\{z\in \C\mid r_c(\theta)<|z|\bigr\}$. We denote by $\psi_c$ this extension and by $R_c(\theta)$ the dynamical ray
\[R_c(\theta):=\psi_c\Big(R_0(\theta)\cap \bigl\{z\in \C\mid r_c(\theta)<|z|\bigr\}\Big).\]
As $r$ tends to $r_c(\theta)$ from above, $\psi_c(r{\rm e}^{2\pi i\theta})$ converges to a point $x\in \C$. If $r_c(\theta)>1$, then $x\in \C\smm K_c$ is an iterated preimage of $0$ and we say that $R_c(\theta)$ bifucates at $x$. If $r_c(\theta)=1$, then $\gamma_c(\theta):=x$ belongs to $K_c$ and we say that $R_c(\theta)$ lands at $\gamma_c(\theta)$. Again, $f_c\circ \gamma_c = \gamma_c\circ \tau$ on the set of $\theta$ such that $R_c(\theta)$ does not bifurcate. In particular,  if $\theta$ is periodic under $\tau$ and $R_c(\theta)$ does not bifurcate, then $\gamma_c(\theta)$ is periodic under $f_c$.

The Multibrot set is connected. The map
\[\phi_{M_{d}}:\C\smm M_{d}\ni c\mapsto\phi_c(c)\in  \C\smm \overline \D\]
is a conformal isomorphism. For $\theta\in \T$, the parameter ray $R_{M_{d}}(\theta)$ is
\[R_{M_{d}}(\theta):= \bigl\{c\in \C\smm M_{d}\mid \arg\bigl(\phi_{M_{d}}(c)\bigr)=2\pi \theta\bigr\}.\]
It is known that if $\theta$ is rational, then as $r$ tends to $1$ from above, $\phi_{M_{d}}^{-1}(r{\rm e}^{2\pi i\theta})$ converges to a point $\gamma_{M_{d}}(\theta)\in M_{d}$. We say that $R_{M_{d}}(\theta)$ lands at $\gamma_{M_{d}}(\theta)$.

\begin{figure}[htbp]

\centering

\includegraphics[width=14.5cm]{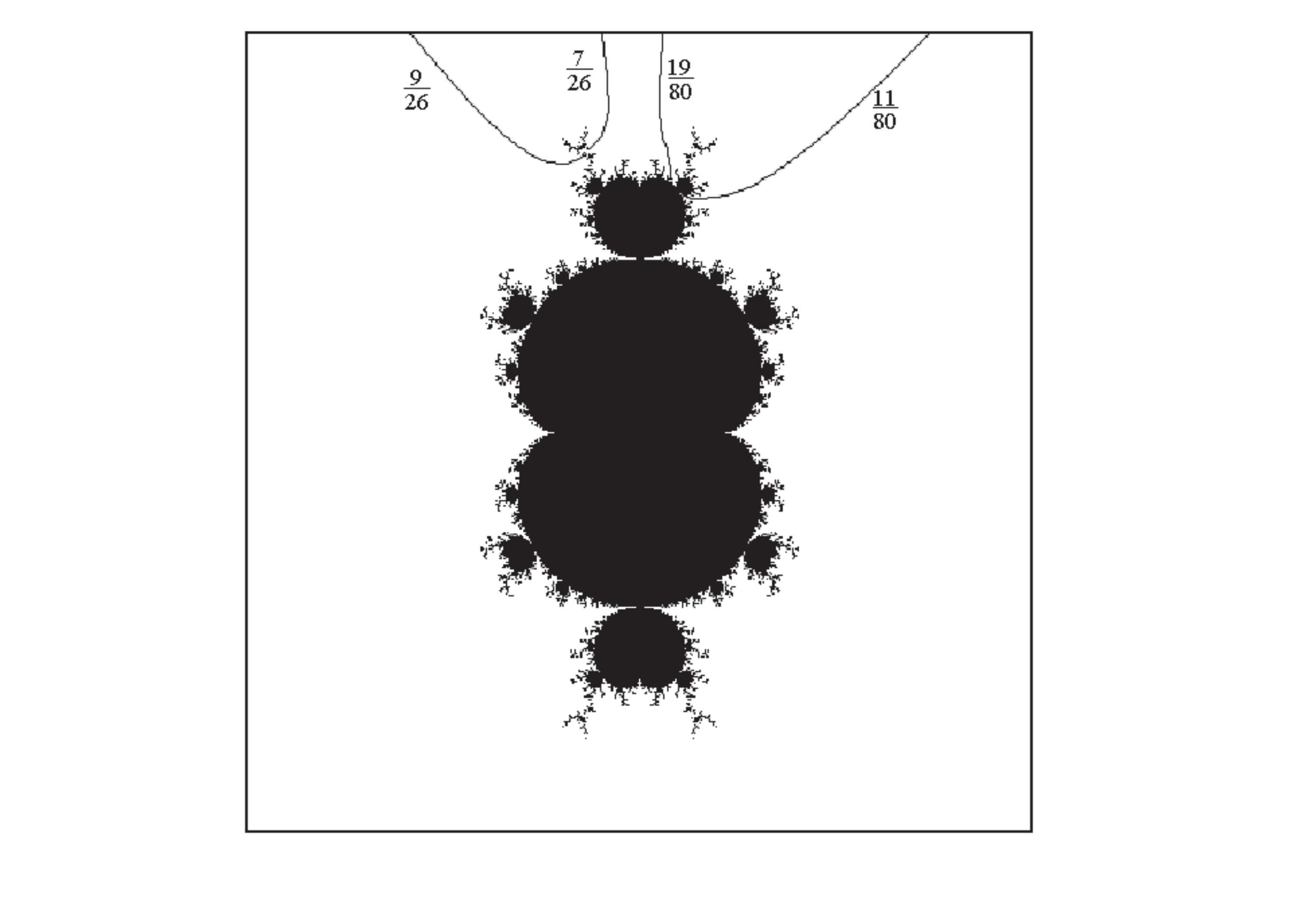}

\caption{ The parameter rays $R_{M_3}(7/26)$ and $R_{M_3}(9/26)$ land on a common root of a primitive hyperbolic component while $R_{M_3}(19/80)$ and $R_{M_3}(11/80)$ land on a common root of a satellite hyperbolic component. Only angles of rays are labelled in the graph.} \label{fig3}

\end{figure}

If $\theta$ is periodic  for $\tau$ of exact period $n$ and if  $c_0:=\gamma_{M_{d}}(\theta)$, then  the point $\gamma_{c_0}(\theta)$ is periodic for $f_{c_0}$ with period $p$ dividing $n$ ($ps=n,\ s\geq1$) and multiplier a $s$-$th$ root of unity.
If the period of $\gamma_{c_0}(\theta)$ for $f_{c_0}$ is exactly $n$ then the multiplier is $1$, $c_0$ is called primitive parabolic parameter, otherwise $c_0$ is called satellite parabolic parameter.

\begin{lemma}[near parabolic map] \label{near parabolic}
$c_0$ is defined as above. When we make a small perturbation to $c_0$ in parameter space, 
If $c_0$ is a primitive parabolic parameter, then the parabolic orbit of $f_{c_0}$ is splitted into a pair of  nearby periodic orbits of $f_c$, both have length $n$;
If $c_0$ is a satellite parabolic parameter, then the parabolic orbit of $f_{c_0}$ is splitted into a pair of  nearby periodic orbits of $f_c$, one has length $p$ and the other has length $sp=n$.
\end{lemma}

This lemma was proved by Milnor in \cite{Mil} lemma $4.2$ for the case $d=2$, but we can translate the proof word by word to the general case.

\begin{figure}[htbp]

\centering
\includegraphics[width=14.5cm]{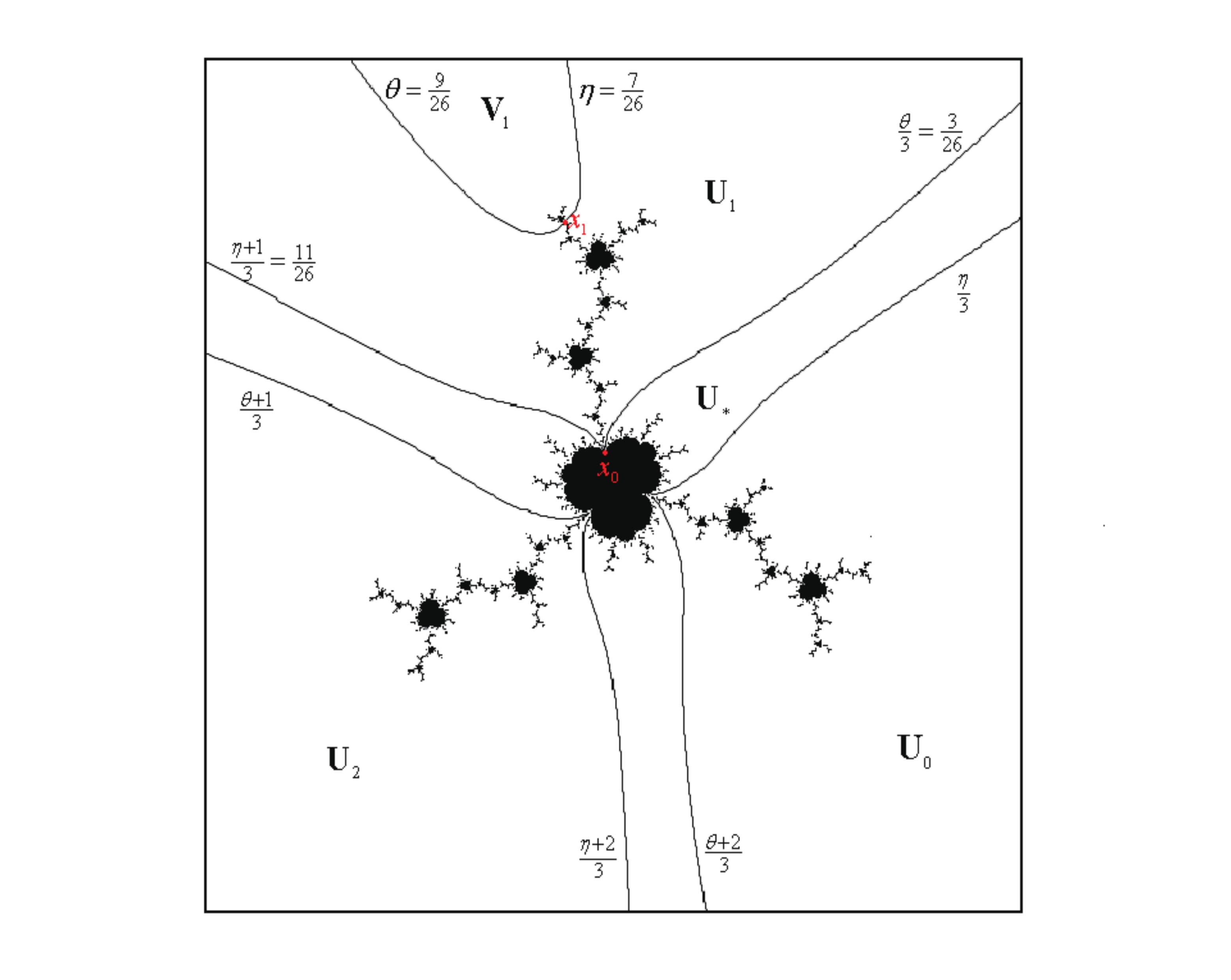}

\caption{ The dynamical plane of $f_{c_0}$. $c_0:=\gamma_{M_3}(7/26)=\gamma_{M_3}(9/26)$ is the root of some primitive hyperbolic component as illustrated in Figure \ref{fig3}. The dynamical rays $R_{c_0}(7/26)$ and $R_{c_0}(9/26)$ land on a common parabolic point of $f_{c_0}$ with period $3$. } \label{fig4}

\end{figure}

Let $H$ be periodic $n(n>1)$ hyperbolic component of $M_{d}$. For every parameter $c\in H, f_{c}$ has an attracting
periodic orbit $\{\ z(c), \ldots, f_{c}^{n-1}(z(c))\ \}$. Its multiplier define a map
$$\mu_{H}: H\rightarrow \D,\  c\mapsto \frac{\partial}{\partial z}f_{c}^{n}(z(c))$$
then $\mu_{H}: H\rightarrow \D$ is $d-1$ covering map with only one branched point .It extends continuously to a
neighborhood of  $\overline{H}$. Considering parameter $c\in \partial H$ such that $\mu_{H}(c)=1$, Eberlein proved that
among these points, there is exactly one $c$ which is the landing point of two parameter rays of period $n$, this point
is called root of $H$ (see Figure \ref{fig3}); the other $d-2$ points are landing points of only one parameter ray of period n each, they are called co-root of $H$ (see Figure \ref{fig6}). $H$ is called primitive or satellite hyperbolic component according to whether its root is  primitive or satellite parabolic parameter.

If $c$ is the root of some hyperbolic component and $c\neq \gamma_{M_{d}}(0)$, then two periodic parameter rays
$R_{M_{d}}(\theta)$ and $R_{M_{d}}(\eta)$ land on $c$, we say $\theta$ and $\eta$ are companion angles, and $\theta,\eta$ have the same period under $\tau$. $c$ is primitive if and only if the orbit of $R_{M_{d}}(\theta)$ and $R_{M_{d}}(\eta)$ under $\tau$ are distinct. In dynamical plane, the dynamic rays $R_{c}(\theta)$ and $R_{c}(\eta)$
land at a common point $x_{1}:=\gamma_{c}(\theta)=\gamma_{c}(\eta)$. This point is on the parabolic orbit
of $f_{c}$ with its immediate basin containing the critical value. $R_{c}(\theta)$ and $R_{c}(\eta)$ are adjacent to the Fatou component containing $c$ and the curve $R_{c}(\theta)\cup R_{c}(\eta)\cup \{x_{1}\}$ is a Jordan curve that cuts the
plane into two connected components: one component, denoted by $V_{1}$, contains the critical value $c$; the other component, denoted by $V_{0}$, contains $R_{c}(0)$ and all  points of parabolic cycle except $x_{1}$. Since $V_{1}$ contains the critical value, its preimage $U_{\star}=f_{c}^{-1}(V_{1})$ is connected and contains the  critical point $0$. It is bounded
by the dynamical rays $R_{c}(\theta/d),\ldots,R_{c}\bigl((\theta+d-1)/d\bigr);\ R_{c}(\eta/d),\ldots,R_{c}\bigl((\eta+d-1)/d\bigr)$.
Suppose $\theta>\eta$, and since each component of $\C\setminus\overline{U_\star}$ is conformally mapped to $V_0$ which is bounded by $R_c(\theta)$ and $R_c(\eta)$, it is easy to see that $R_{c}\bigl((\theta+k-1)/d\bigr)$ and $R_{c}\bigl((\eta+k)/d\bigr)$
land on a common point which is one of the preimage of $x_{1}$ for $k\in \mathbb{Z}_{d}$.
Denote $U_{k}$ the component of $\mathbb{C}\setminus R_{c}\bigl((\theta+k-1)/d\bigr)\cup \{\gamma_{c}\bigl((\eta+k)/d\bigr)\}\cup R_{c}\bigl((\eta+k)/d\bigr)$ disjoint with $U_{\star}$. See Figure \ref{fig4} (primitive case) and Figure \ref{fig5} (satellite case). Note that
$f_{c}:U_{k}\rightarrow V_{0}$ is conformal.

\begin{figure}[htbp]

\centering

\includegraphics[width=14.5cm]{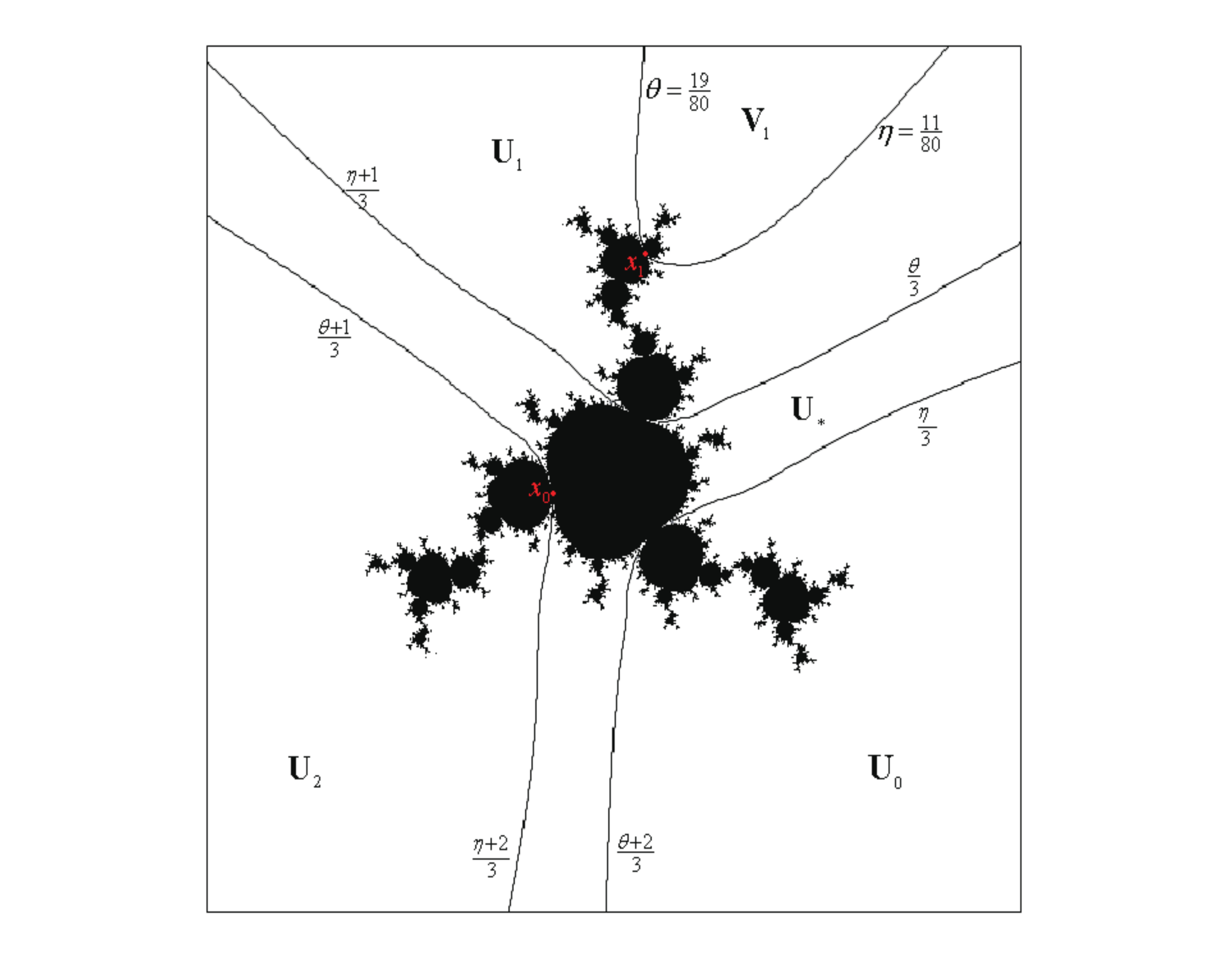}

\caption{ The dynamical plane of $f_{c_1}$. $c_1:=\gamma_{M_3}(11/80)=\gamma_{M_3}(19/80)$ is the root of some satellite hyperbolic component as illustrated in Figure \ref{fig3}. The dynamical rays $R_{c_1}(11/80)$ and $R_{c_1}(19/80)$ land on a common parabolic point of $f_{c_1}$ with period $2$. } \label{fig5}

\end{figure}

If $c$ is a co-root of some hyperbolic component, then  exactly one period parameter ray $R_{M_{d}}(\beta)$ land
on it (see Figure \ref{fig6}). In dynamical plane, $R_{c}(\beta)$ is the unique dynamical ray landing on a parabolic periodic point $\gamma_{c}(\beta):=x_{1}$, whose immediate basin contains the critical value $c$. The parameter $c$ is a primitive parabolic parameter.
Denote $V_{1}$ the union of Fatou component containing $c$ and external ray $R_{c}(\beta)$, $V_{0}=\mathbb{C}\setminus \overline{V_{1}}$,
$U_{\star}=f_{c}^{-1}(V_{1})$. $U_{k}$ is the component of $f_{c}^{-1}(V_{0})$
adjacent with $R_{c}\bigl((\beta+k-1)/d\bigr)$ and $R_{c}\bigl((\beta+k)/d\bigr), k\in \mathbb{Z}_{d}$.(see Figure \ref{fig7}).

Remark: in our paper, if $c$ is a parabolic parameter, then $f_{c}$ has unique parabolic orbit, denoted by
$\{x_{0},x_{1},\ldots ,x_{p-1}\}$. $x_{1}$ is the point whose immediate basin contains critical value $c$.

The following lemma provides a criterion for $\theta$ such that $\gamma_{M_{d}}(\theta)$ is a primitive parabolic parameter.
\begin{definition}
Let $\theta$ be a periodic angle of period $n$ and the d-expansion of $\theta$ be $.\overline{\epsilon_1\ldots\epsilon_n}$. We call $\epsilon_1\ldots\epsilon_n$ the periodic part
of the d-expansion of $\theta$.
\end{definition}

\begin{figure}[htbp]

\centering

\includegraphics[width=11.5cm]{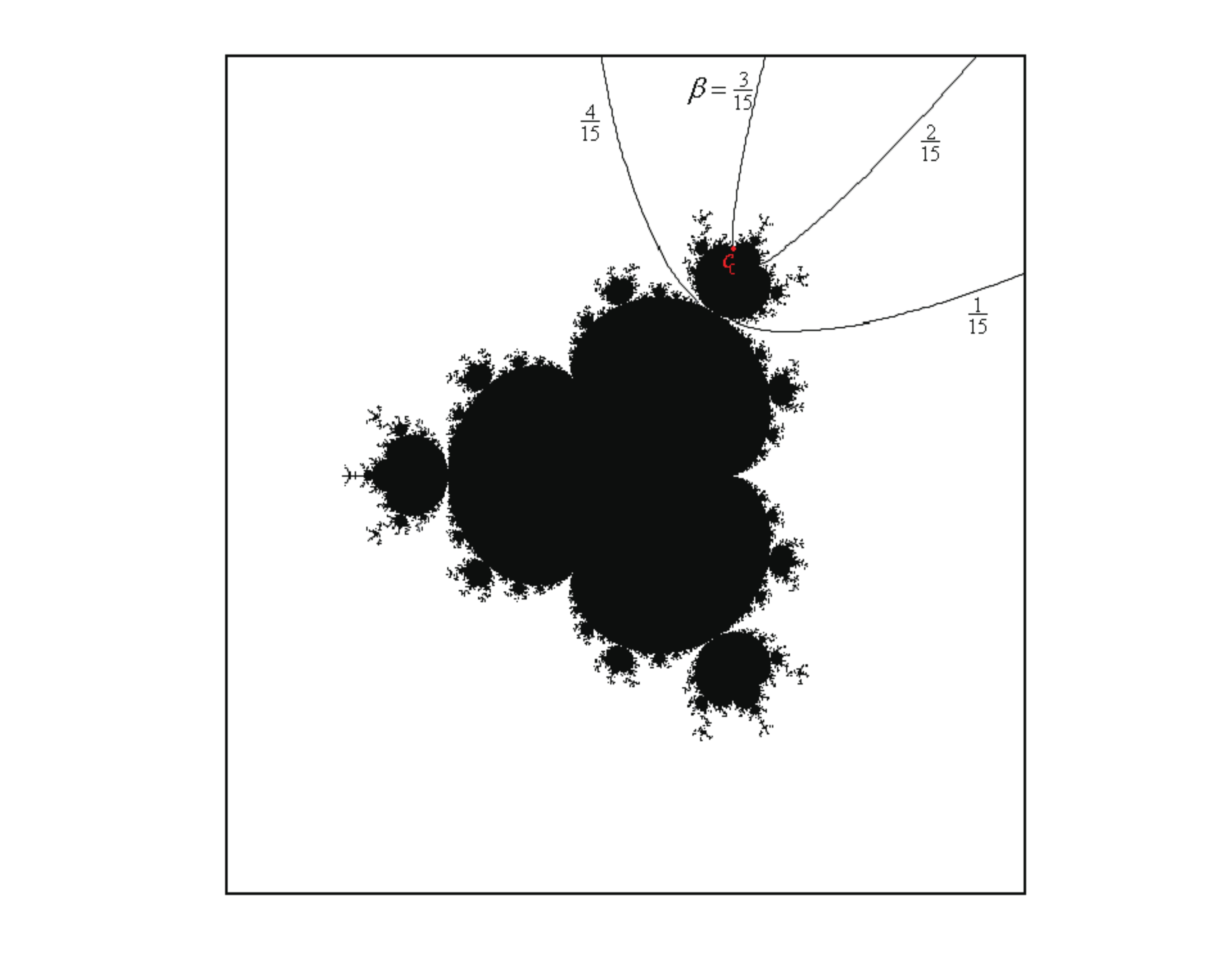}

\caption{Multibrot set $M_4$. The parameter rays $R_{M_4}(1/15)$ and $R_{M_4}(4/15)$ land on the root of some hyperbolic component. $R_{M_4}(2/15)$ and $R_{M_4}(1/5)$ land on two co-root of this hyperbolic component respectively.} \label{fig6}

\end{figure}

\begin{figure}[htbp]

\centering
\includegraphics[width=12.5cm]{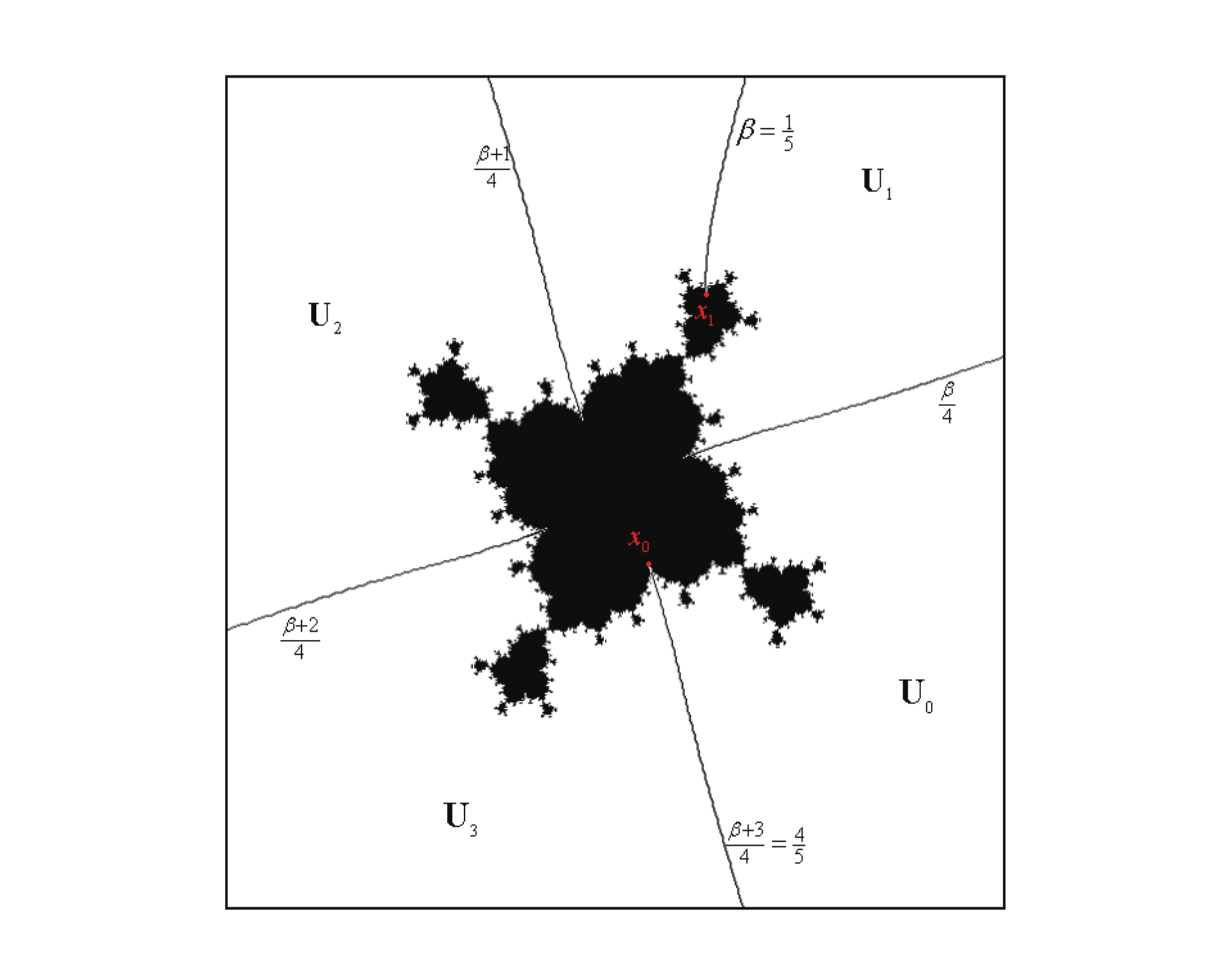}

\caption{The dynamical plane of $f_{c_0}$. $c_0:=\gamma_{M_4}(1/5)$ is a co-root of the hyperbolic component illustrated in Figure \ref{fig6}. $R_{c_0}(1/5)$ is the unique dynamical ray landing on $\gamma_{c_0}(1/5)$ which
is the parabolic point of $f_{c_0}$ with period $2$.} \label{fig7}

\end{figure}

\begin{lemma}\label{criterion}
$\theta$ is periodic under $\tau$ with period $n\geq2$. If $c_0:=\gamma_{M_{d}}(\theta)$ is the root of some satellite hyperbolic component, then $\theta$ satisfies the following properties:
\begin{description}
  \item[(1)] $\nu(\theta)$ is cyclic.
  \item[(2)]  Denote by $\overline{w^{s-1}w_\star}$ the cyclic expression of $\nu(\theta)$ where $w=\nu_1\ldots\nu_t$, $t$ is a proper factor of $n$ and $ts=n$. Then the last digit of the period part of the $d$-expansion of $\theta$ is $\nu_t$ or $\nu_t-1$.
\end{description}
Moreover, if $\theta$ is maximal in its orbit, then $\nu(\theta)$ also satisfies
\begin{description}
  \item[(3)] $t$ is the length of parabolic orbit  and the last digit of the period part of the $d$-expansion of $\theta$ must be $\nu_t-1\in [0,d-2]$.
\end{description}
\end{lemma}

\begin{proof}
Let $\eta$ be the companion angle of $\theta$, then in dynamical plane of $f_{c_0}$, $R_{c_0}(\theta)$ and $R_{c_0}(\eta)$ land on $x_1$ (see Figure \ref{fig5}). As $V_1$ contains no points and external rays of the parabolic orbit, then $\{x_0,x_1,\ldots,x_{p-1}\}$ together with their external rays belong to $\bigcup_{k=0}^{d-1}\overline{U}_k$.

For $c_0$ is satellite parabolic parameter, the length $p$ of parabolic orbit is a proper factor of $n$ and $f_{c_0}$ acts on the rays of the orbit transitively. Then we have, in $\nu(\theta)=\overline{\nu_1\ldots\nu_{n-1}\star}$, $\nu_j=\nu_{j(\text{mod}) p}$ for $1\leq j\leq n-1$, that is,
$\nu(\theta)=\overline{u^{l-1}u_\star}$ where $u=\nu_1\ldots\nu_p$. By definition of kneading sequence, we can see
$\tau^{\circ(p-1)}(\theta)\in \bigl( (\theta+\nu_p-1)/d,\ (\theta+\nu_p)/d\bigr)$. It follows $x_0$ together with its external rays belong to $\overline{U}_{\nu_p}$. Then  $\tau^{n-1}(\theta)$ is either $(\theta+\nu_p-1)/d\ (\theta>\eta)$ or $(\theta+\nu_p)/d\ (\theta<\eta)$\ (see Figure \ref{fig9}). So the last digit of $d$-expansion of $\theta$ is either $\nu_p-1\ (\theta>\eta)$ or $\nu_p\ (\theta<\eta).$ Let $w=\nu_1\ldots\nu_t$ be the primitive root of $u$, then $u=w^{p/t}$.  We have $\overline{w^{s-1}w_\star}$ is the cyclic expression of $\nu(\theta)$ (proposition \ref{cyclic}) and $\nu_t=\nu_p$, so $\theta$ satisfies property $(1)$ and $(2)$.

\begin{figure}[htbp]

\centering
\includegraphics[width=14.5cm]{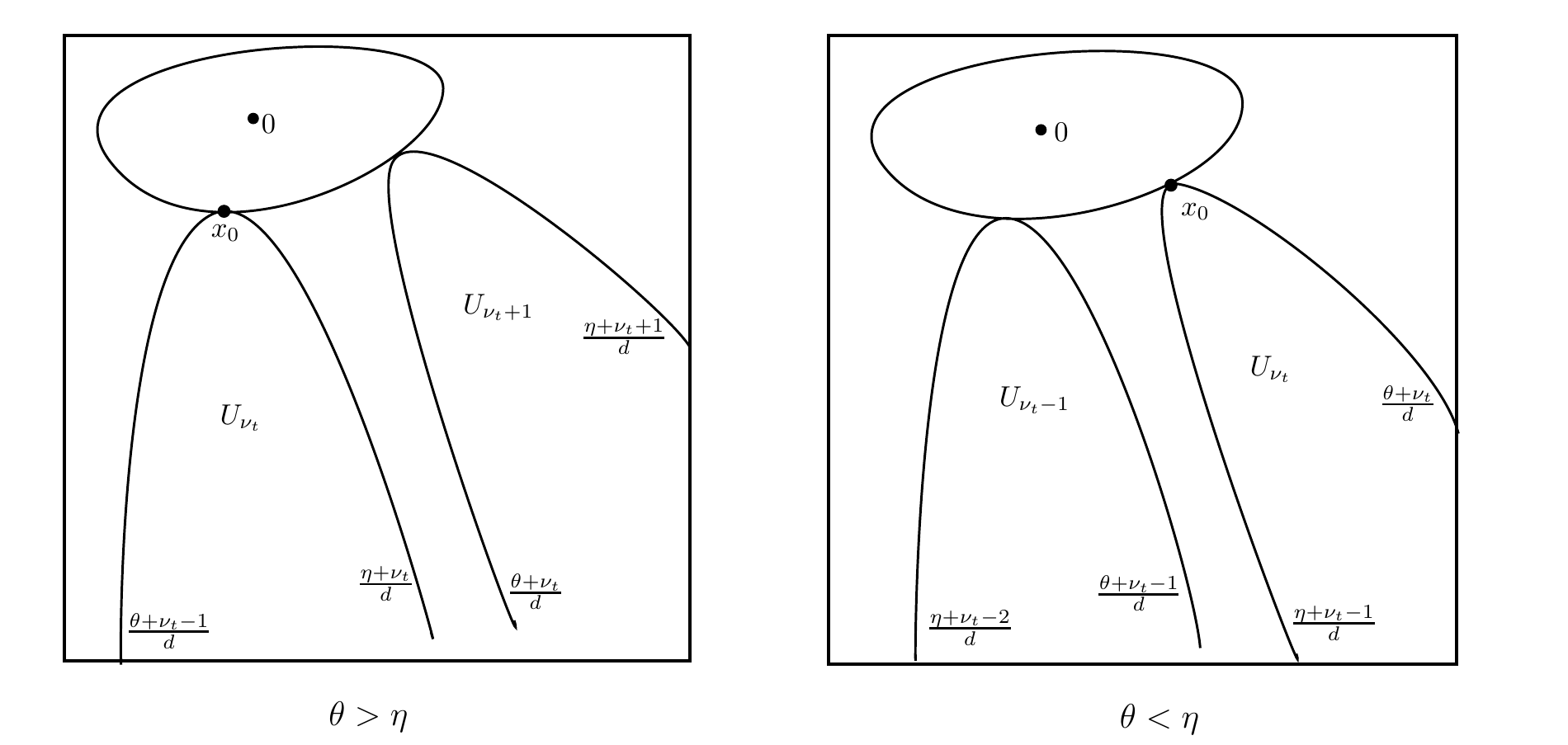}

 \caption{}\label{fig9}

\end{figure}

Furthermore, if $\theta$ is maximal in its orbit, then $\theta>\eta$, so the last digit of the period part of the $d$-expansion of $\theta$ must be $\nu_t-1$. By lemma \ref{key}, $\theta=.\overline{w^{s-1}\nu_1\ldots\nu_{t-1}(\nu_t-1)}$ and $0\leq\nu_t-1\leq d-2$. Note that the angles of external rays belonging to $x_1$ are $\theta,\ \tau^{p}(\theta),\ldots,\tau^{(s-1)p}(\theta)$ with the order $\theta>\tau^{p}(\theta)>\cdots>\tau^{(s-1)p}(\theta).$ The maximum of $\theta$ implies $\eta$ is the second largest angle in orbit of $\theta$, then $\eta=\tau^{p}(\theta)=.\overline{u^{l-2}\nu_1\ldots\nu_{p-1}(\nu_p-1)u}$. If $u$ is not primitive, then $p/t>1$. It follows
$\tau^{t}(\theta)>\tau^{p}(\theta)=\eta$, a contradiction to that $\eta$ is the second largest angle in orbit of $\theta$. So
$u$ is a primitive word and hence $t=p$ is length of parabolic orbit.

\end{proof}

 Then once $\theta$ doesn't satisfy the property in this lemma, we have $\gamma_{M_{d}}(\theta)$ is a primitive parabolic parameter. The lemma below can be seen as a application of lemma \ref{criterion}.

\begin{lemma}\label{apply}
Assume $\theta=.\overline{w^{s-1}\nu_1\ldots\nu_{t-1}(\nu_t-1)}$ is maximal in its orbit, where $w=\nu_1\ldots\nu_t$ is primitive with
$\nu_t\in[1,d-1]$ and $t$ is a proper factor of $n$ with $ts=n$. Let \[\beta_{v_t-i}=.\overline{w^{s-1}\nu_1\ldots\nu_{t-1}(\nu_t-i)}\ \text{ for } 2\leq i\leq \nu_t\]
\[\beta_{-1}=
\begin{cases}
.\overline{w^{s-1}\nu_1\ldots(\nu_{t-1}-1)(d-1)}&\mbox{as }t\geq 2 \\
.\overline{k\ldots k(k-1)(d-1)}&\mbox{as }t=1
\end{cases}
\]
Then $\gamma_{M_{d}}(\beta_{\nu_t-i})$ is a primitive parabolic parameter for any $2\leq i\leq \nu_t$. $\gamma_{M_{d}}(\beta_{-1})$ is a satellite parabolic parameter for $\theta=.\overline{(d-1)\cdots(d-1)(d-2)}$ and a primitive
parabolic parameter for any other case.
\end{lemma}

\begin{proof}
Let $\beta=.\overline{w^{s-1}\nu_1\ldots\nu_{t-1}j}$ be any angle among $\{\beta_{\nu_t-i}\}_{2\leq i\leq \nu_t}$, then
$0\leq j\leq \nu_t-2$. The maximum of $\theta$ implies  the maximum of $\beta$ in its orbit. Since $w$ is primitive, by lemma \ref{key}, we have $\overline{w^{s-1}w_\star}$ is the cyclic expression of $\nu(\beta)$.
As $j\leq \nu_t-2<\nu_t-1$, with the maximum of $\beta$, the property $(3)$ in lemma \ref{criterion} is not satisfied.
So $\gamma_{M_{d}}(\beta)$ is a primitive parabolic parameter.

For $\beta_{-1}$, the maximum of $\theta$ implies $\beta_{-1}$ is greater than $\tau(\beta_{-1}),\ \tau^{2}(\beta_{-1}),\ldots, \tau^{n-2}(\beta_{-1})$ but less than $\tau^{n-1}(\beta_{-1})$. It follows
$\nu(\beta)=\begin{cases}
\overline{w^{s-1}\nu_1\ldots\nu_{t-1}\star}&\mbox{as }t\geq 2 \\
\overline{k\ldots k k\star}&\mbox{as }t=1
\end{cases}=\overline{w^{s-1}w_\star}$.
It is  the cyclic expression of $\nu(\beta)$,  then if $\beta$
satisfies the property in lemma \ref{criterion}, $\nu_t$ is either $0$ or $d-1$. Since $1\leq \nu_t\leq d-1$, we have
$\nu_t$ must be $d-1$, then the maximum of $\theta$ implies $\theta=.\overline{(d-1)\cdots(d-1)(d-2)}$. So $\gamma_{M_{d}}(\beta_{-1})$ is a primitive parabolic parameter as long as $\theta\neq.\overline{(d-1)\cdots(d-1)(d-2)}$.
In the case of $\theta=.\overline{(d-1)\cdots(d-1)(d-2)}$, we will see in lemma \ref{3.6} that $\gamma_{M_{d}}(\theta)$ is the root of a hyperbolic component attached to the main cardioid and $\beta_{-1}$ is the
companion angle of $\theta$. In this case, $\gamma_{M_{d}}(\beta_{-1})$ is a satellite parabolic parameter.

\end{proof}

{\noindent\bf Remark.} In this lemma, we distinguish $\beta_{-1}$ according to whether $t\geq 2$ or $t=1$. It is because that we don't find a uniform expression of $\beta_{-1}$ for the two cases rather than the case of $t=1$ is special.

\subsection{Itineraries outside the Multibrot set}

\begin{figure}[htbp]

\centering
\includegraphics[width=13.5cm]{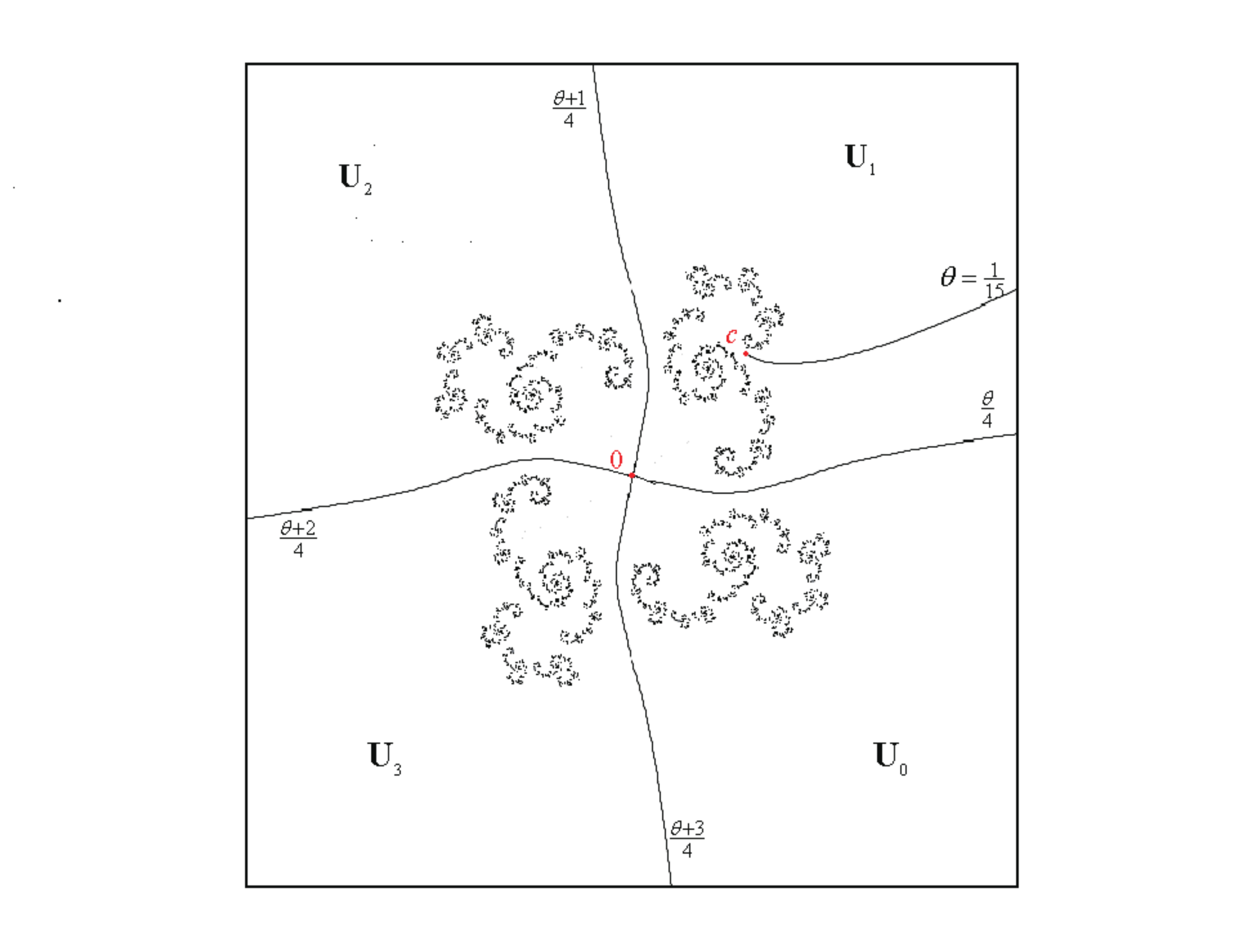}
\caption{The regions $U_0,\ U_1,\ U_2,\ U_3$ for a parameter $c$ belonging to $R_{M_4}(1/15)$.} \label{fig8}
\end{figure}

If $c\in \C\smm M_{d}$, the Julia set of $f_c$ is a Cantor set. If $c\in R_{M_{d}}(\theta)$ with $\theta\neq 0$ not necessarily periodic, then the dynamical rays $R_c(\theta/d)\ldots R_c\bigl((\theta+d-1)/d\bigr)$ bifurcate on the critical point. The set $R_c(\theta/d)\cup \ldots \cup R_c\bigl((\theta+d-1)/d\bigr)\cup \{0\}$ separates the complex plane in $d$ connected components. We denote by $U_0$ the component containing the dynamical ray $R_c(0)$ and by $U_1,\ldots,U_{d-1}$ the other component in counterclockwise (see Figure \ref{fig8}).

The orbit of a point $x\in K_c$ has an itinerary with respect to this partition. In other words, to each $x\in K_c$, we can associate a sequence $\iota_c(x)\in \{0,1,\ldots d-1\}^{\N}$ whose $j$-th term is equal to $k$ if $f_c^{\circ j-1}(x)\in U_k$ . A point $x\in K_c$ is periodic for $f_c$ if and only if the itinerary $\iota_c(x)$ is periodic for the shift with the same period.

The map $\iota_c:K_c\to \{0,1,\ldots d-1\}^{\N}$ is a bijection. In particular, for each itinerary $\iota\in \{0,\ldots,d-1\}^{\N}$ and each $c\in \C\smm\bigl(M_{d}\cup R_{M_{d}}(0)\bigr)$, there is a unique point $x(\iota,c)\in K_c$ whose itinerary is $\iota$. For a given $\iota\in \{0,\ldots,d-1\}^{\N}$, the map
$\C\smm\bigl(M_{d}\cup R_{M_{d}}(0)\bigr)\longrightarrow \C\quad c \mapsto x(\iota,c)\in \C$
is continuous, and even holomorphic (as can be seen by applying the Implicit Function Theorem).

\begin{proposition}\label{permuting1}
Let $\overline{\ep_1\ldots\ep_{n-1}\star}$ be the kneading sequence of a periodic angle $\theta$ with period $n\geq2$. If $c_{0}:=\gamma_{M_{d}}(\theta)$ is a primitive parabolic parameter and if one follows continuously the periodic points of period $n$ of $f_c$ as $c$ makes a small turn around $c_0$, then the periodic points with itineraries
$\overline{\ep_1\ldots\ep_{n-1}k}$ and $\overline{\ep_1\ldots\ep_{n-1}(k+1)}$ get exchanged where $k\in\Z_{d}$
is the last digit of the period part of the $d$-expansion of $\theta$.
\end{proposition}

\begin{proof}
Since $c_0$ is  a primitive parabolic parameter,  then the periodic point $x_1:=\gamma_{c_0}(\theta)$ has period $n$ and multiplier $1$. According to Case $2$ in the proof of smoothness and lemma \ref{near parabolic}, the projection from a small neighborhood of $(c_0,x_1)$ in $X_n$ to the first coordinate is a degree $2$  covering. So the neighborhood of $(c_0,x_1)$ in $\overline{X_n}$ can be written as
\[\bigl\{(c_0+\delta^2,x(\delta)),(c_0+\delta^2,x(-\delta))\ \big|\ |\delta|<\varepsilon\bigr\}\]
where $x:(\C,0)\to (\C,x_1)$ is a holomorphic germ with $x'(0)\neq 0$. In particular, the pair of periodic points for $f_c$ which are splitted from $x_1$  get exchanged when $c$ makes a small turn around $c_0$.
So, using analytic continuation on $\C\setminus(M_d\cup R_{M_d}(0))$, it is enough to show that there exists a $c\in \C\smm M_{d}$ close to $c_0$ such that $x(\pm\sqrt{c-c_0})$ have itineraries $\overline{\ep_1\ldots\ep_{n-1}k}$ and $\overline{\ep_1\ldots\ep_{n-1}(k+1)}$ where $k\in\Z_{d}$
is the last digit of the period part of the $d$-expansion of $\theta$.

Let us denote by $V_0(c_0)$, $V_1(c_0)$, $U_0(c_0),\ldots,U_{d-1}(c_0)$ and $U_\star(c_0)$ the sets defined in the previous section. For $j\geq 0$, set $x_j:= f_{c_0}^{j}(x_0)$ and observe that for $j\in [1,n-1]$, we have $x_j\in U_{\ep_j}(c_0)$.

 For $c\in R_{M_{d}}(\theta)$, consider the following compact subsets of the Riemann sphere :
 \[R(c):= R_c(\theta)\cup\{c,\infty\}\quad\text{and}\quad S(c):=R_c(\theta/d)\cup\ldots\cup R_c\bigl((\theta+d-1)/d\bigr)\cup\{0,\infty\}.\]
Denote
by $U_0(c)$ the component of $\C\smm S(c)$ containing $R_c(0)$ and by $U_1(c),\ldots,U_{d-1}(c)$ the other component in
counterclockwise. From any sequence $\{c_m\}\subset R_{M_d}(\theta)$ converging to $c_0$, by extracting a subsequence if necessary,
we can assume $R(c_m)$ and $S(c_m)$ converge respectively,  for the Hausdorff topology on compact subsets of $\C\cup \{\infty\}$, to connected compact sets $R$ and $S$. Since $S(c)=f_c^{-1}\bigl(R(c)\bigr)$, we have $S=f_{c_0}^{-1}(R)$.
According to \cite[Section 2 and 3]{PR}, $R\cap (\C\smm K_{c_0}) = R_{c_0}(\theta)$,
 the intersection of $R$ with the boundary of $K_{c_0}$ is reduced to $\{x_1\}$ and the intersection of $R$ with the interior of $K_{c_0}$ is contained in the immediate basin of $x_1$, whence in $V_1$.
It follows  $R\subset \overline{V_{1}}(c_{0})$ and $S\subset \overline{U}_{\star}(c_0)$, that means  any compact subset of $\C\smm \overline U_\star(c_0)$ is contained in $\C\smm S(c_{m})$ for $m$ sufficiently large .

For $j\in [1,n-1]$ and let $D_j$ be a sufficiently small disk around $x_j$ so that
\[\overline D_j\subset U_{\ep_j}(c_0)\subset \C\smm \overline U_\star(c_0).\]  According to the previous discussion, if $m$ is sufficiently large, we have
\[f_{c_m}^{ j-1}\bigl(x(\pm\sqrt{c_m-c_0})\bigr)\subset D_j\subset U_{\ep_j}(c_{m}).\]
So the first $n-1$ symbols of the itineraries of $x(\pm\sqrt{c_m-c_0})$ are all $\epsilon_1,\ldots ,\epsilon_{n-1}.$
As $x(\sqrt{c_m-c_0})$ and $x(-\sqrt{c_m-c_0})$ are different $n$ periodic points of $f_{c_m}$,  their itineraries must be
different. It follows $f_{c_m}^{ n-1}\bigl(x(\pm\sqrt{c_m-c_0})\bigr)$, which are splitted  from $x_0$, lie in different component of $\C\setminus S(c_{m})$. Combining with the fact that $R_{c_0}\bigl((\theta+k)/d\bigr)$ lands on $x_0$ ($k$ is the last digit of the period part of the $d$-expansion of $\theta$),
we have $f_{c_m}^{ n-1}\bigl(x(\pm\sqrt{c_m-c_0})\bigr)$ belong to $U_k(c_m)$ and $U_{k+1}(c_m)$ respectively, then
$x(\pm\sqrt{c_m-c_0})$ have itineraries $\overline{\ep_1\ldots\ep_{n-1}k}$ and $\overline{\ep_1\ldots\ep_{n-1}(k+1)}$
respectively.

\end{proof}

\begin{lemma}\label{3.6}

For $\theta=1-1/(d^n-1)=.\overline{(d-1)\cdots (d-1)(d-2)}\ (n\geq2)$, we have $\gamma_{M_{d}}(\theta)$  is the root of some periodic $n$ hyperbolic component attached to the main cardioid. If $\eta$ is denoted the companion angle of $\theta$, then
$\eta=d\theta-d+1$.

\end{lemma}

\begin{proof}

Let $c_0:=\gamma_{M_d}(\theta)$, then $x_1:=\gamma_{c_{0}}(\theta)$ is the parabolic periodic point of $f_{c_0}$ as previous. By lemma \ref{key}, $\nu(\theta)=\overline{(d-1)\cdots(d-1)\star}$, so $\overline{(d-1)\cdots(d-1)\star}$ is the cyclic expression of $\nu(\theta)$. If $x_0\neq x_1$, then the length of parabolic orbit is greater than $1$. It implies the property $(3)$ in lemma \ref{criterion} is not satisfied, so $c_0$ is a primitive parabolic parameter. According to proposition\ref{permuting1}, when $c\in \C\setminus M_d $ is close to $c_0$,
      $x_1$ splits into two $n$ periodic point $y,z$ of $f_c$ with itineraries $\overline{(d-1)\cdots(d-1)(d-2)}$ and $\overline{(d-1)\cdots(d-1)(d-1)}$. It leads to a contradiction to the period $n$ of $y$ and $z$. So $x_0=x_1$ and then
      $c_0$ is the  root of some periodic $n$ satellite hyperbolic component attached to the main cardioid.

 By the maximum of $\theta$, we have $U_{d-1}$ is bounded by $R_{c_0}\bigl((\theta+d-2)/d\bigr)$ and $R_{c_0}\bigl((\eta+d-1)/d\bigr)$. $\nu(\theta)=\overline{(d-1)\cdots(d-1)\star}$ implies $R_{c_0}(\theta)\subset \overline{U}_{d-1}$,
then $\theta\leq (\eta+d-1)/d$ and $x_0$ is on the boundary of $U_{d-1}$. On the other hand, $(\eta+d-1)/d$ is in the
orbit of $\theta$, so $\theta\geq(\eta+d-1)/d$. Then we have $\eta=d\theta-d+1$.

\end{proof}

{\noindent\bf Remark.} The dynamical rays $R_{c_0}(\theta)$ and $R_{c_0}(\eta)$ are consecutive among the rays landing at $x_0$.
Lemma \ref{3.6} implies  $R_{c_0}(\theta)$ is mapped to $R_{c_0}(\eta)$. It follows that each dynamical ray landing at
$x_0$ is mapped to the one which is once further clockwise.\\

\begin{proposition}\label{permuting2}
Let $\theta=1-1/(d^n-1)=.\overline{(d-1)\cdots(d-1)(d-2)}$ be periodic with period $n\geq 2$. If one follows continuously the periodic points of period $n$ of $f_c$ as $c$ makes a small turn around $\gamma_{M_{d}}(\theta)$, then the periodic points in the cycle of $\iota_c^{-1}(\overline{(d-1)\cdots(d-1)(d-2)})$ get permuted cyclically.
\end{proposition}

\begin{figure}[htbp]

\centering
\includegraphics[width=14.5cm]{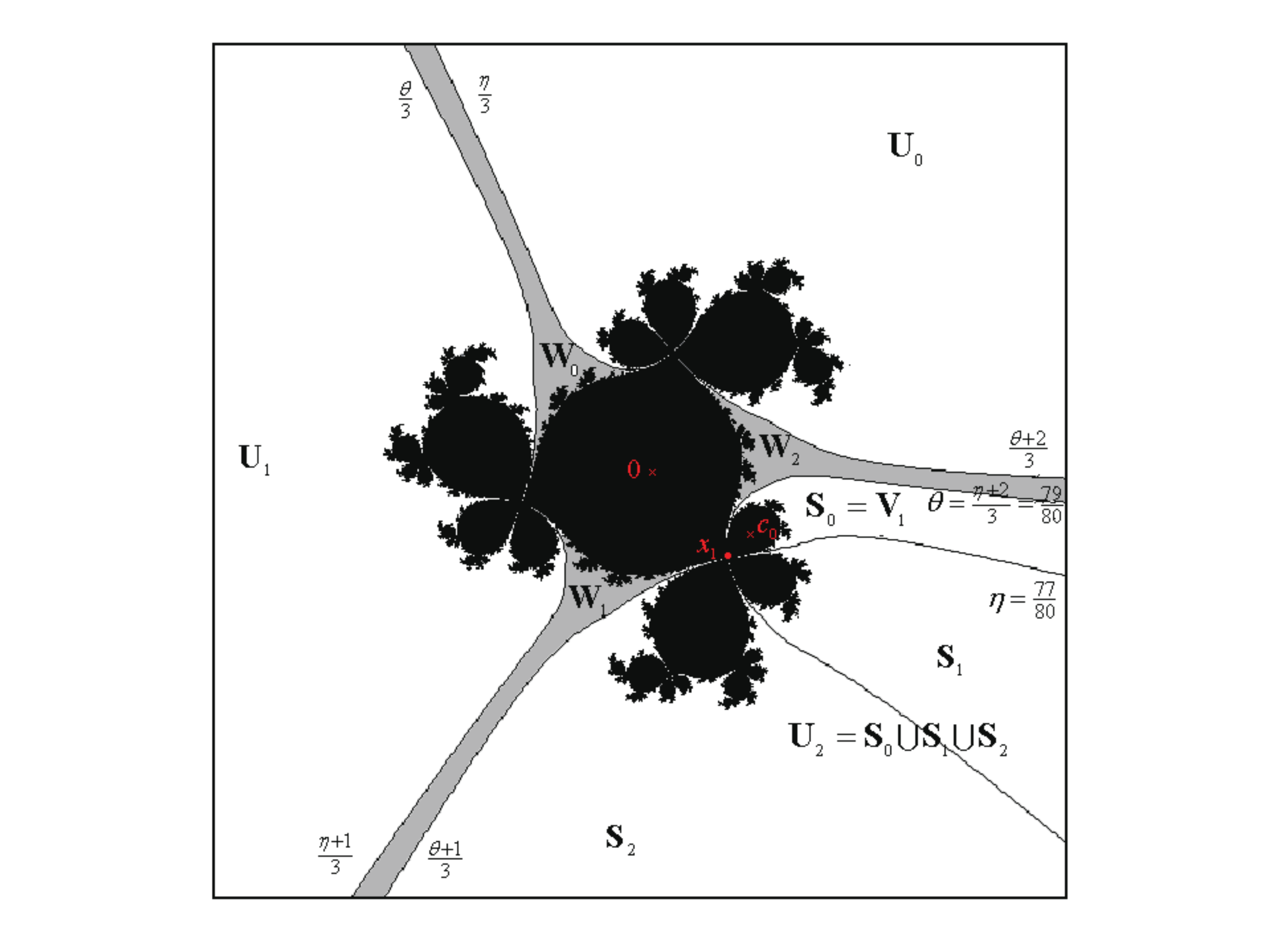}
\caption{The dynamical plane of $f_{c_0}$. $c_0:=\gamma_{M_3}(\theta)$ with $\theta=.\overline{2221}$ } \label{fig10}
\end{figure}

\begin{proof}
Set $c_0:=\gamma_{M_d}(\theta)$. By Lemma \ref{3.6}, all the dynamical rays
$R_{c_0}\bigl(\tau^j(\theta)\bigr)$ land on a common fixed point $x_0$. This fixed point is parabolic and the companion angle of $\theta$, denoted by $\eta$, equals to $d\theta-(d-1)\equiv d\theta (\text{mod}\ \Z)$. $V_1(c_0)\subset U_{d-1}(c_0)$\ which is bounded by $R_{c_0}\bigl((\theta+d-2)/d\bigr)$\ and $R_{c_0}(\theta)$.

According to Case $3$ in the proof of smoothness and lemma \ref{near parabolic}, we have the projection from a small neighborhood of $(c_0,x_0)$ in $X_n$ to the parameter plane is a degree $n$ covering. Then the  neighborhood of $(c_0,x_0)$
in $\overline{X_n}$ can be written as
\[\bigl\{(c_0+\delta^n,x(\delta)),(c_0+\delta^n,x(\omega\delta)),\ldots,(c_0+\delta^n,x(\omega^{n-1}\delta))\ \big|\ |\delta|<\varepsilon\bigr\}\]
where $x:(\C,0)\to (\C,x_0)$ is a holomorphic germ satisfying $x'(0)\neq 0$.
So, for $c$ close to $c_0$, the set $x\{\sqrt[n]{c-c_0})\}$ is a cycle of period $n$ of $f_c$, and
when $c$ makes a small turn around $c_0$, the periodic points in the cycle $x\{\sqrt[n]{c-c_0})\}$ get permuted cyclically. So, combining with analytic continuation on $\C\setminus(M_d\cup R_{M_d}(0))$, it is enough to show there exists a $c\in \C\smm M_d$ close enough to $c_0$ such that the point $\iota_c^{-1}(\overline{(d-1)\cdots(d-1)(d-2)})$ belongs to $x\{\sqrt[n]{c-c_0}\}$.
Equivalently, we must show that there is a sequence $\{c_j\}\subset \C\smm M_d$ converging to $c_0$, such that the periodic point $y_j:=\iota_{c_j}^{-1}(\overline{(d-1)\cdots(d-1)(d-2)})$  converges to $x_0$.

Let $\{c_j\}\subset R_{M_d}(\theta)$ converge to $c_0$ as $j\rightarrow \infty$. Without loss of generality, we may assume that the sequence $y_j$ converges to a point $z$, $R(c_j)$\ converges to $R$\ and $S(c_j) $ converges to $S $ in Hausdoff topology. The definition of
$R(c)$, $S(c)$, $U_0(c),\ldots,U_{d-1}(c)$ are in the proof of proposition \ref{permuting1}. As $(c_0, z)$ is on  $\overline{X_n}$,
then $z $ is either the parabolic fixed point or repelling $n $ periodic point of $f_{c_0}$.

Suppose $z $ is
a repelling $n $ periodic point, set $z_i:=f_{c_0}^{i}(z)$. Now we will define a new sequence of open domain
$\bigl\{ W_k(c_0) \bigr\}$. $W_k(c_0) $ is  the connected component of $U_\star(c_0)\setminus$\ the closure of Fatou component
containing $0 $, adjacent with $U_k(c_0),U_{k+1}(c_0)$ (see Figure \ref{fig10}). According to \cite[Section 2 and 3]{PR}, $R\cap (\C\smm K_{c_0}) = R_{c_0}(\theta)$,
 the intersection of $R$ with the boundary of $K_{c_0}$ is reduced to $\{x_0\}$ and the intersection of $R$ with the interior of $K_{c_0}$ is contained in the immediate basin of $x_0$. It follows
 $\{ z_0,\ldots,z_{n-1}\}\bigcap S=\emptyset$.
 Then for $j$ sufficiently large, $\{ z_0,\ldots,z_{n-1}\}\subset \C\setminus S_{c_j}$. As $y_j $ has itineraries $\overline{(d-1)\cdots(d-1)(d-2)}$,
we have $\{z_0,\ldots z_{n-2}\}\subset U_{d-1}(c_0)\bigcup \overline{W}_{d-1}(c_0)$, $z_{n-1}\in U_{d-2}(c_0)
\bigcup \overline{W}_{d-2}(c_0) $.

{\noindent\bf Claim 1.} $z_{n-1}\notin \overline{W}_{d-2}(c_0)$.

{\noindent\rm Proof}. In $J(f_{c_0})$, $x_0$\ is the unique periodic point with more than one external rays landing on it (refer to \cite[proposition 3.3]{Poi}). So there is exactly one external ray landing on $z_{n-1}$ with period $n$. Its angle is denoted by $\dfrac{a}{d^n-1},\ a$\ is a integer.
If $z_{n-1}\in \overline{W}_{d-2}$, the angle of external ray belonging to $z_{n-1}$ satisfy
\[ \frac{\eta+d-2}{d}<\frac{a}{d^n-1}<\frac{\theta+d-2}{d}\ \ (\ \theta=1-\frac{1}{d^n-1},\ \eta=d\theta-d+1\ ).\]
by simple computation, we have
\[ \frac{k(d^n-1)}{d-1}-d^{n-1}-1+\frac{1}{d}<a<\frac{k(d^n-1)}{d-1}-d^{n-1},\]
a contradiction to $a$\ is an integer. This ends the proof of claim 1.

{\noindent\bf Claim 2.} $z_{n-1}\notin U_{d-2}(c_0)$.

{\noindent\rm Proof}. If $z_{n-1}\in U_{d-2}(c_0)$, we label the sectors at $x_0$ by $S_i (0\leq i\leq n-1)$ clockwise with $S_0=V_1(c_0)$. The dynamics between these sectors satisfy
\[V_1(c_0)=S_0\xrightarrow{f_{c_0}}S_1
   \xrightarrow {f_{c_0}}
   \cdots
   \xrightarrow {f_{c_0}} S_{n-2}\xrightarrow{f_{c_0}} S_{n-1}=\C\setminus \overline{U}_{d-1}(c_0) \]
As $\{z_0,\ldots z_{n-2}\}\subset U_{d-1}(c_0)\bigcup \overline{W}_{d-1}(c_0)$, we have $z_0=f_{c_0}(z_{n-1})$ belongs
to the union of $\overline{W_{d-1}}(c_0)$ and $\bigcup_{i=1}^{n-2}S_i$.
If $z_0\in S_{i_0}\ (1\leq i_0\leq n-2)$, then $f_{c_0}^{(n-2-i_0)}(z_0)=z_{n-2-i_0}\in f_{c_0}^{(n-2-i_0)}(S_{i_0})=S_{n-2}$. It follows $f_{c_0}(z_{n-2-i_0})=z_{n-1-i_0}$ must
belong to $\overline{W}_{d-1}(c_0)$. So $z_{n-i_0}\in S_0$ and $f_{c_0}^{(i_0-1)}(z_{n-i_0})=z_{n-1}\in f_{c_0}^{(i_0-1)}(S_0)=S_{i_0-1}$, contradiction to $z_{n-1}\in U_{d-2}$. If $z_0\in \overline{W}_{d-1}(c_0)$, then $z_1\in S_0$.
We have $f_{c_0}^{(n-2)}(z_1)=z_{n-1}\in f_{c_0}^{(n-2)}(S_0)=S_{n-2}$, also a contradiction to $z_{n-1}\in U_{d-2}(c_0)$. This ends the proof of claim 2.

The two claim imply the assumption that $z$ is repelling $n$ periodic point is false and then $z$\ must be a parabolic fixed point of $f_{c_0}$,
that is $z=x_0$.
\end{proof}

\subsection{Proof of Theorem \ref{main}}

Fix $n>1$ (the case $n=1$ has been treated directly at the beginning). We proceed to show that $X_n$ is connected.

Set $X:=\C\setminus \bigl(M_d\cup R_{M_d}(0)\bigr)$ and $F_n:=\C\setminus$ all the landing points of periodic $n$ parameter rays.
Take any pair of points $(a,w), (a',w')$\ in $X_n$. By analytic continuation, we may assume $a,a'\in X$. Again by analytic continuation on simply connected open set $X$, we may assume $a=a'$. Thus it is enough to show that there exists a loop in $F_n$\ based on $a$\ such that the analytic continuation along the loop connects $w$\ and $w'$. We will give a algorithm to find such a loop.

Let z be any $n$\ periodic point of $f_a$.
\begin{description}

  \item[step 1] In the orbit of $z$, there is a point with maximal itineraries among the shift of $\iota_a(z)$\ in
  the lexicograph order, denoted by $\overline{\epsilon_1\ldots\epsilon_n}$. Set $\theta=.\overline{\epsilon_1\ldots\epsilon_n}$ ($\theta$ is maximal in its orbit). If $\theta$ satisfies the properties in  lemma \ref{criterion}, do step $2$\ below. Otherwise, $\gamma_{M_d}(\theta)$ is a primitive parabolic parameter. According to lemma \ref{key} and proposition \ref{permuting1}, when $a$\ makes a turn around $\gamma_{M_d}(\theta)$, the periodic point of $f_a$
  with itineraries $\overline{\epsilon_1\ldots\epsilon_n}$\ and $\overline{\epsilon_1\ldots(\epsilon_n+1)}$\ get changed. Then $z$\ is connected to a new orbit containing $\iota_a^{-1}(\overline{\epsilon_1\ldots(\epsilon_n+1)})$.
  For this new orbit, repeat doing step $1$.

  \item[step 2] $\theta=.\overline{\epsilon_1\ldots\epsilon_n}$ is maximal in its orbit and satisfies the properties in lemma \ref{criterion}. If $\theta=.\overline{(d-1)\cdots(d-1)(d-2)}$, step $2$ ends. Otherwise, let $\overline{w^{s-1}w_\star}$ be the cyclic expression of $\nu(\theta)$ where $w=\nu_1\ldots \nu_t,\ \nu_t\in[1,d-1]$. As in lemma \ref{apply}, we obtain a sequence of angles $\{\beta_{\nu_t-2},\ldots,\beta_0,\ \beta_{-1}\}$ and know that $\gamma_{M_d}(\beta_{\nu_t-i})$ is a primitive parabolic parameter with $\nu(\theta)=\overline{\epsilon_1\ldots\epsilon_{n-1}\star}$ for any $i\in [2,\nu_t+1]$. Then by proposition \ref{permuting1} again, as $a$ makes a turn around $\gamma_{M_d}(\beta_{\nu_t-i})\ (2\leq i\leq\nu_t+1)$, the periodic points of $f_a$ with itineraries $\overline{\epsilon_1\ldots\epsilon_{n-1}(\nu_t-i)}$ and $\overline{\epsilon_1\ldots\epsilon_{n-1}(\nu_t-i+1)}$ get changed. Then let $a$ makes turns around from $\gamma_{M_d}(\beta_{\nu_t-2})$ to $\gamma_{M_d}(\beta_{-1})$ one by one, we have $\iota_a^{-1}(\overline{\epsilon_1\ldots \epsilon_{n-1}\epsilon_n})$ are connected with $\iota_a^{-1}(\overline{\epsilon_1\ldots \epsilon_{n-1}(d-1)})$ by analytic continuation through the points $\iota_a^{-1}(\overline{\epsilon_1\ldots \epsilon_{n-1}(\epsilon_{n}-1)}),\ldots,\iota_a^{-1}(\overline{\epsilon_1\ldots \epsilon_{n-1}0})$. For the new
      periodic point $\iota_a^{-1}(\overline{\epsilon_1\ldots \epsilon_{n-1}(d-1)})$, do step $1$.

\end{description}

Every time a $n$ periodic point of $f_a$ passes though step $1$ or step $2$, the sum of all digits in the itineraries of the output periodic point is greater than that of the input one. For fixed $n$, this sum  is bounded $\bigl(\text{the bound is } (d-1)n-1 \bigr)$, then each $n$ periodic point $z$ can be connected to the orbit containing $\iota_a^{-1}(\overline{(d-1)\cdots(d-1)(d-2)})$.

In our case, applying the procedure above to $w$\ and $w'$, we have $w$\ and $w'$\ are connected to two points of
the periodic orbit containing $\iota_a^{-1}(\overline{(d-1)\cdots (d-1)(d-2)})$. Proposition \ref{permuting2}\ tells us, by
analytic continuation, any two point in this orbit can be connected as long as $a$\ makes the appropriate number of turns around $\gamma_{M_d}(1-\frac{1}{d^n-1})$. Thus $w$\ and $w'$\ are connected.
\qed


\begin{thebibliography}{BKM}
\bibitem[BT]{BT} X. Buff and Tan Lei, The quadratic dynatomic curve are smooth and irreducible. Volume special en
l'honneur des 80 ans de Milnor a paraitre.
\end{thebibliography}
\end{document}